\documentclass[12pt, a4paper]{article}
\usepackage{times}
\usepackage{booktabs}
\usepackage{pifont}
\usepackage{floatrow}
\usepackage{caption}
\usepackage[fleqn]{amsmath}
\usepackage{amsfonts,amsthm,amssymb,mathrsfs,bbding}
\usepackage{txfonts}
\usepackage{graphics,multicol}
\usepackage{graphicx}
\usepackage{color}
\usepackage{caption}
\usepackage{cite}
\usepackage{latexsym,bm}
\usepackage{indentfirst}
\usepackage[colorlinks=true,anchorcolor=blue,filecolor=blue,linkcolor=blue,urlcolor=blue,citecolor=blue]{hyperref}
\usepackage{extarrows}
\usepackage{mathtools}
\usepackage{authblk}
\usepackage{enumerate}

\usepackage{xcolor}
\usepackage[normalem]{ulem}

\newtheorem{theorem}{Theorem}[section]

\newtheorem{prob}{Problem}[section]
\newtheorem{claim}{Claim}

\newtheorem{lemma}{Lemma}[section]

\newtheorem{case}{Case}
\newtheorem{subcase}{Case}[case]

\theoremstyle{definition}

\addtocounter{section}{0}

\begin{document}

\title{The exact Tur\'{a}n number of the even wheel $W_{2k+2}$ among non-$3$-partite graphs \footnote{Supported by the National Natural Science Foundation of China {(No. 12371361)} and Distinguished Youth Foundation of Henan Province {(No. 242300421045)}.}}

\author[1]{Qixuan Yuan}
\author[1]{Ruifang Liu\thanks{Corresponding author. E-mail address: rfliu@zzu.edu.cn (R. Liu).}}
\author[2]{Sanming Zhou}
\affil[1]{\small School of Mathematics and Statistics, Zhengzhou University, Zhengzhou, Henan 450001, China}
\affil[2]{\small School of Mathematics and Statistics, The University of Melbourne, Parkville, VIC 3010, Australia}

\date{}

\maketitle
\openup 0.6 \jot 

\begin{abstract}
Let $\mathrm{ex}(n,H)$ denote the Tur\'{a}n number of $H$. A graph is color-critical if there exists an edge $e\in E(H)$ such that $\chi(H-e)<\chi(H)$. For a color-critical graph $H$ with $\chi(H)=r+1$, Simonovits' chromatic critical edge theorem implies that there exists an $n_0(H)$ such that $\mathrm{ex}(n,H)=e(T_{n,r})$ and the Tur\'{a}n graph $T_{n,r}$ is the only extremal graph provided $n\geq n_0(H).$ Let $W_{2k+2}$ be the even wheel obtained by joining a vertex to a cycle of length $2k+1,$ where $k\geq1$ is an integer. Since $W_{2k+2}$ is color-critical and $\chi(W_{2k+2})=4$, $T_{n,3}$ is the unique extremal graph for $W_{2k+2}$-free graphs of sufficiently large $n.$ Note that the extremal graph $T_{n,3}$ is 3-partite. In this paper, we determine the exact Tur\'{a}n number of $W_{2k+2}$ in non-$3$-partite graphs and characterize all extremal graphs provided $n$ is sufficiently large.

\smallskip
\textit{Keywords}: Tur\'{a}n number, Non-3-partite, Even wheel, Extremal graphs

\smallskip
\textit{AMS Classification}: 05C50, 05C35
\end{abstract}

\section{Introduction}

All graphs considered in this paper are finite, undirected and simple. For a graph $G$, let $V(G)$ and $E(G)$ denote its vertex set and edge set, respectively, and call $|V(G)|$ and $e(G)=|E(G)|$ the \emph{order} and \emph{size} of $G$, respectively. For a vertex $v\in V(G)$, let $N_G(v)$ be the set of neighbors of $v$ in $G$, and let $d_G(v)=|N_G(v)|$ be the \emph{degree} of $v$ in $G$. For $u\in V(G)$ and $S\subseteq V(G)$, let $N_G(u,S)$ be the set of neighbors of $u$ in $S$, and set $d_G(u,S) = |N_G(u,S)|$. For two disjoint subsets $A,B\subseteq V(G),$ we use $E_G(A,B)$ to denote the set of edges of $G$ with one end-vertex in $A$ and the other in $B$, and set $e_G(A,B) = |E_G(A,B)|$. For a partition $\{V_1, V_2, V_3\}$ of $V(G)$, set
$$
e_G(V_1,V_2,V_3)=\sum_{1\leq i<j\leq 3}e_G(V_i,V_j).
$$
Denote by $G[S]$ the subgraph of $G$ induced by $S$ and by $G\setminus S$ the subgraph of $G$ induced by $V(G)\setminus S$. We use $G \vee H$ to denote the join of two graphs $G$ and $H$, $\overline{H}$ the complement of $H$, and $\chi(H)$ the chromatic number of $H$. A graph is \emph{$r$-partite} if its vertex set can be partitioned into $r$ nonempty parts (each called a \emph{partite set}) such that no two vertices in the same part are adjacent. A graph is \emph{non-$r$-partite} if such a partition does not exist. The complete graph with order $n$ is denoted by $K_n$ and the cycle with length $n \ge 3$ is denoted by $C_n$.

For a given graph $H$, a graph $G$ is called {\it $H$-free} if it contains no subgraph isomorphic to $H$. The {\it Tur\'{a}n number} $\mathrm{ex}(n, H)$ is the maximum size of an $H$-free graph with order $n$. An $H$-free graph with order $n$ and size $\mathrm{ex}(n,H)$ is called an {\it extremal graph}, and the family of such extremal graphs is denoted by $\mathrm{Ex}(n,H)$. The Tur\'{a}n-type problem relative to a given forbidden graph $H$, which asks for the exact value of $\mathrm{ex}(n,H)$ and the structure of $\mathrm{Ex}(n,H)$, is one of the central problems in extremal graph theory. The {\it Tur\'{a}n graph} $T_{n, r}$ is the unique complete $r$-partite graph with order $n$ such that the sizes of its $r$ partite sets are as equal as possible. The study of Tur\'{a}n-type problems dates back to Mantel's theorem \cite{Mantel1907} which is stated as follows.

\begin{theorem}[\!\cite{Mantel1907}]\label{th1.1}
$\mathrm{ex}(n, C_3)=\big\lfloor\frac{n^{2}}{4}\big\rfloor$ and $\mathrm{Ex}(n, C_3)=\{T_{n, 2}\}$.
\end{theorem}

Denote by $\mathrm{ex}_{\chi(H)}(n, H)$ the maximum size of a non-$(\chi(H)-1)$-partite $H$-free graph with order $n$, and by $\mathrm{Ex}_{\chi(H)}(n, H)$ the family of such extremal graphs. Note that $T_{n,2}$ is bipartite. In \cite{E1955}, Erd\H{o}s proved the following non-bipartite version of Mantel's theorem.

\begin{theorem}[\!\cite{E1955}]
$\mathrm{ex}_3(n, C_3)=\big\lfloor\frac{(n-1)^{2}}{4}\big\rfloor+1.$
\end{theorem}

In \cite{CJ2002}, Caccetta and Jia extended Erd\H{o}s' theorem by showing that if $G$ is a non-bipartite graph on $n$ vertices containing no odd cycles of length at most $2k+1,$ then $e(G)\leq \big\lfloor\frac{(n-2k+1)^2}{4}\big\rfloor+2k-1.$ Moreover, they characterized all extremal graphs. For longer odd cycles, Bondy \cite{Bondy1971,B1971}, F\"{u}redi and Gunderson \cite{Furedi2015}, and Woodall \cite{Woodall1972} generalized Theorem \ref{th1.1} from $C_{3}$ to $C_{2k+1}$.

\begin{theorem}[\!\cite{Bondy1971,B1971,Furedi2015,Woodall1972}]\label{th1.3}
Let $k\geq2$ and $n\geq4k-2$ be integers. Then
$$\mathrm{ex}(n, C_{2k+1})=\left\lfloor\frac{n^{2}}{4}\right\rfloor~~\mbox{and}~~\mathrm{Ex}(n, C_{2k+1})=\{T_{n,2}\}.$$
\end{theorem}

Ren, Wang, Wang and Yang \cite{Ren2024} obtained the following non-bipartite refinement of Theorem \ref{th1.3}. Denote by $T_{n-2,2}\bullet C_3$ the graph obtained by identifying a vertex of $T_{n-2,2}$ with a vertex of $C_3$.

\begin{theorem}[\!\cite{Ren2024}]
Let $k\geq2$ and $n\geq318k$ be integers. Then
$$\mathrm{ex}_3(n, C_{2k+1})=\left\lfloor\frac{(n-2)^{2}}{4}\right\rfloor+3~~\mbox{and}~~\mathrm{Ex}_3(n, C_{2k+1})=\{T_{n-2,2}\bullet C_3\}.$$
\end{theorem}

In 1941, Tur\'{a}n \cite{Turan1941} generalized Mantel's theorem from triangles to complete graphs.

\begin{theorem}[\!\cite{Turan1941}]
Let $n \ge r \ge 1$ be integers. Then
$$
\mathrm{ex}(n, K_{r+1})=e(T_{n, r})\ \text{ and }\ \mathrm{Ex}(n, K_{r+1})=\{T_{n, r}\}.
$$
\end{theorem}

Since $T_{n,r}$ is $r$-partite, the maximum size of a non-$r$-partite $K_{r+1}$-free graph is not given by $e(T_{n, r})$, and a natural problem is to determine the maximum size of a non-$r$-partite $K_{r+1}$-free graph. In this direction, Brouwer \cite{B1981} proved the following refinement of Tur\'an's theorem.

\begin{theorem}[\!\cite{B1981}]\label{th1.6}
Let $n\geq2r+1$ be an integer. Then
$$\mathrm{ex}_{r+1}(n, K_{r+1})=e(T_{n,r})-\left\lfloor\frac{n}{r}\right\rfloor+1.$$
\end{theorem}

Theorem \ref{th1.6} was also independently studied in several other references, see \cite{Amin2013,KP,TU}.

A graph is {\it color-critical} if it contains an edge whose removal reduces its chromatic number. Simonovits \cite{S1968,S1974} established the following fundamental result, which is known as Simonovits' chromatic critical edge theorem.

\begin{theorem}
[\!\cite{S1968,S1974}]\label{th1.7}
If $H$ is a color-critical graph with $\chi(H)=r+1,$ then there exists a positive integer $n_0(H)$ such that for any $n\geq n_0(H)$ we have $\mathrm{ex}(n, H)=e(T_{n,r})$ and $\mathrm{Ex}(n, H)=\{T_{n,r}\}$.
\end{theorem}

Since the unique extremal graph $T_{n,r}$ in Theorem \ref{th1.7} is $r$-partite, it is natural to study the non-$r$-partite version of the Tur\'{a}n problem: determine the maximum size $\mathrm{ex}_{r+1}(n, H)$ of a non-$r$-partite $H$-free graph together with the family $\mathrm{Ex}_{r+1}(n, H)$ of extremal graphs. A generalized {\it theta graph} $\theta(l_1,l_2,\ldots,l_t)$ is obtained from internally disjoint paths of lengths $l_1,l_2,\ldots,l_t,$ respectively, by sharing a common pair of endpoints. In 2021, Zhai, Fang and Shu \cite{Zhai2021} showed that $\theta(p,q,s)$ is a color-critical graph with chromatic number three for any $p,q,s$ which are not all of the same parity. Moreover, they determined $\mathrm{Ex}(n,\theta(p,q,s))=\{T_{n,2}\}$ for $n\geq 9(p+q+s-1)^2-3(p+q+s-1).$ For fixed integers $q,s\geq 2$ with even $q$, $\theta(1,q,s)$ is a color-critical graph with chromatic number three. Fang and Lin \cite{Fang2025} determined all the graphs in $\mathrm{Ex}_3(n,\theta(1,q,s))$ for sufficiently large $n.$ A {\it book graph} $B_{r+1}$ consists of $r+1$ triangles sharing a common edge, where $r\geq0$. Clearly, $B_{r+1}$ is color-critical and $\chi(B_{r+1})=3$. By Theorem \ref{th1.7}, $T_{n,2}$ is the unique extremal graph for $B_{r+1}$-free graphs of sufficiently large order $n$. Moreover, Edwards \cite{E1977} and independently Khad\v{z}iivanov and Nikiforov \cite{KN} confirmed Erd\H{o}s' booksize conjecture and showed that $\mathrm{ex}(n,B_{r+1})=e(T_{n,2})$ for $n\geq6r$. Since the extremal graph $T_{n,2}$ is bipartite, the non-bipartite version of this problem is again natural. For $r=0$, the graph $B_{r+1}$ is a triangle, and the problem reduces to the theorem of Erd\H{o}s mentioned above. For $r\geq1$ and sufficiently large $n,$ Miao, Liu and van Dam \cite{Miao2026} recently proved that $\mathrm{ex}_3(n,B_{r+1})= \big\lfloor\frac{(n-1)^2}{4}\big\rfloor+2r$ and characterized all extremal graphs. For $r\geq 3$ and $k\geq 1,$ {\it the generalized book graph} $B_{r,k}$ is defined as the join $K_r \vee \overline{K}_k$ of $K_r$ and the complement $\overline{K}_k$ of $K_k$. Note that $B_{r,k}$ is a color-critical graph with chromatic number $r+1.$ Yu and Li \cite{Yu2025} determined $\mathrm{ex}_{r+1}(n,B_{r,k})$ and characterized all extremal graphs for sufficiently large $n.$

The \emph{wheel} $W_t$ with order $t \ge 4$ is the graph obtained from the cycle $C_{t-1}$ by adding a new vertex and joining it to every vertex on this cycle; the new vertex becomes the center vertex and the cycle $C_{t-1}$ is the rim of $W_t.$ Wheels form a classical and important family of forbidden graphs in Tur\'an-type problems. As such they have been studied extensively. The Tur\'an-type problem for odd wheels was resolved in \cite{Dzido2018} and \cite{Yuan2021}, as covered in the next two theorems.

\begin{theorem}[\!\cite{Dzido2018}]
Let $W_5$ be the wheel on $5$ vertices. Then
\begin{eqnarray*}
\mathrm{ex}(n,W_5)=
\begin{cases}
\frac{n^2}{4}+\frac{n}{2}-1, &n\equiv 2\!\!\!\!\!\pmod{4},\\[2mm]
\big\lfloor \frac{n^2}{4}\big\rfloor+\big\lfloor \frac{n}{2}\big\rfloor, &otherwise.
\end{cases}
\end{eqnarray*}
\end{theorem}

\begin{theorem}[\!\cite{Yuan2021}]
Let $k\geq3$ and $n$ be integers. For $n$ sufficiently large,
\begin{eqnarray*}
\mathrm{ex}(n,W_{2k+1})=\max\left\{n_0n_1+\left\lfloor\frac{(k-1)n_0}{2}\right\rfloor+1:n_0+n_1=n,\ n_0,n_1\geq 0\ \text{ are integers }\right\}.
\end{eqnarray*}
\end{theorem}

Regarding even wheels, Dzido \cite{Dzido2013} proved that $\mathrm{ex}(n,W_{2k})=\big\lfloor\frac{n^2}{3}\big\rfloor$ for every $k\geq3$ and $n\geq6k-10.$ In addition, he showed that $\mathrm{ex}(n,W_6)=\big\lfloor\frac{n^2}{3}\big\rfloor$ for all $n\geq 6.$
Note that $W_{2k+2}$ is color-critical and $\chi(W_{2k+2})=4$. Hence Simonovits' chromatic critical edge theorem implies that
\begin{eqnarray*}
\mathrm{ex}(n,W_{2k+2})=e(T_{n,3}) \quad\text{and}\quad \mathrm{Ex}(n,W_{2k+2})=\{T_{n,3}\}
\end{eqnarray*}
for sufficiently large $n.$ Since the extremal graph $T_{n,3}$ is $3$-partite, it is natural to study the following problem.

\begin{prob}\label{pro1.1}
Determine the maximum size of a non-$3$-partite $W_{2k+2}$-free graph with order $n$ together with the corresponding family of extremal graphs.
\end{prob}

This problem has been settled in the case $k=1$ (see Theorem \ref{thm:K4} below). Since $W_4\cong K_4,$ this case is equivalent to the extremal problem for non-$3$-partite $K_4$-free graphs.
For positive integers $a_1,\ldots,a_5,$ let $C_5[a_1,\ldots,a_5]$ denote the graph obtained from $C_5$ by replacing the $i$-th vertex with an independent set of size $a_i$, and by replacing
each edge of $C_5$ with a complete bipartite graph between the corresponding two independent sets. If $n \ge 6$ is even, then let $\mathscr{H}_n=\left\{C_5[\frac{n}{2}-2,t,1,1,\frac{n}{2}-t]: 1\leq t\leq \frac{n}{2}-1 \right\} \cup \left\{C_5[\frac{n}{2}-1,t,1,1,\frac{n}{2}-t-1]: 1\leq t\leq \frac{n}{2}-2 \right\}.$ If $n \ge 5$ is odd, then let $\mathscr{H}_n=\left\{C_5[\frac{n-1}{2}-1,t,1,1,\frac{n-1}{2}-t]:1\leq t\leq \frac{n-1}{2}-1\right\}.$

\begin{theorem}[\!\cite{Amin2013}]
	\label{thm:K4}
Let $n\geq 7$ and $n=3q+r,$ where $r\in\{0,1,2\}.$ Then
\begin{eqnarray*}
\mathrm{ex}_4(n,W_4)=\frac{n^2}{3}- \frac{n}{3} +\frac{r(r+2)}{6}-\frac{r}{2}+1=e(T_{n,3})-\left\lfloor\frac{n}{3}\right\rfloor+1.
\end{eqnarray*}
and
\begin{eqnarray*}
\mathrm{Ex}_4(n,W_4)=
\begin{cases}
\left\{C_5\vee \overline K_{n-5}\right\}, & q=2,\\[2mm]
\left\{G\vee \overline K_{q+1}: G\in\mathscr H_{2q-1}\right\}\cup \left\{G\vee\overline K_q: G\in\mathscr H_{2q}\right\}, & q\ge 3,\ r=0,\\[2mm]
\left\{G\vee\overline K_{q+1}: G\in\mathscr H_{2q}\right\}\cup \left\{G\vee\overline K_q: G\in\mathscr H_{2q+1}\right\}, & q\ge 3,\ r=1,\\[2mm]
\left\{G\vee\overline K_{q+1}: G\in\mathscr H_{2q+1}\right\}, & q\ge 3,\ r=2.
\end{cases}
\end{eqnarray*}
\end{theorem}

In this paper, we solve Problem \ref{pro1.1} for all $k \geq 2$ and sufficiently large $n$. Let $K_{|V_1|,|V_2|,|V_3|}$ be the complete $3$-partite graph on $n$ vertices with partite sets $V_1, V_2$ and $V_3$, where $|V_1|=\big\lceil\frac{n-1}{3}\big\rceil+1$ or $|V_1|=\big\lfloor\frac{n-1}{3}\big\rfloor+1$ and $\big||V_2|-|V_3|\big|\leq 1.$
Let $H$ be the graph obtained from $K_{|V_1|,|V_2|,|V_3|}$ by adding an edge $uv$ with $\{u,v\}\subset V_1.$
Let $V^*_2\subset V_2$ and $w\in V_3,$ where $1\leq |V^*_2|\leq |V_2|-1.$ Let $G^*(n,|V^*_2|)$ be the graph obtained from $H$ by removing all edges in $E(V_2\setminus V^*_2,\{v\})\cup E(V_3\setminus \{w\},\{u\})\cup E(V^*_2,\{w\})$ (see Fig. \ref{fig1}). Define
$$
\mathscr{G}(n) = \{G^*(n,|V^*_2|): 1\leq |V^*_2|\leq |V_2|-1\}.
$$
Note that, for any $G^*\in \mathscr{G}(n)$ and $x\in V(G^*)$, $G^*[N(x)]$ is bipartite and so contains no odd cycles. Hence $G^*$ is $W_{2k+2}$-free.
Note also that $e(G^*(n,|V^*_2|))= e(T_{n,3})-\left(n-\big\lceil\frac{n}{3}\big\rceil\right)+2=e(T_{n-1,3})+2.$

\begin{figure}[H]
\centering
\includegraphics[width=0.6\textwidth]{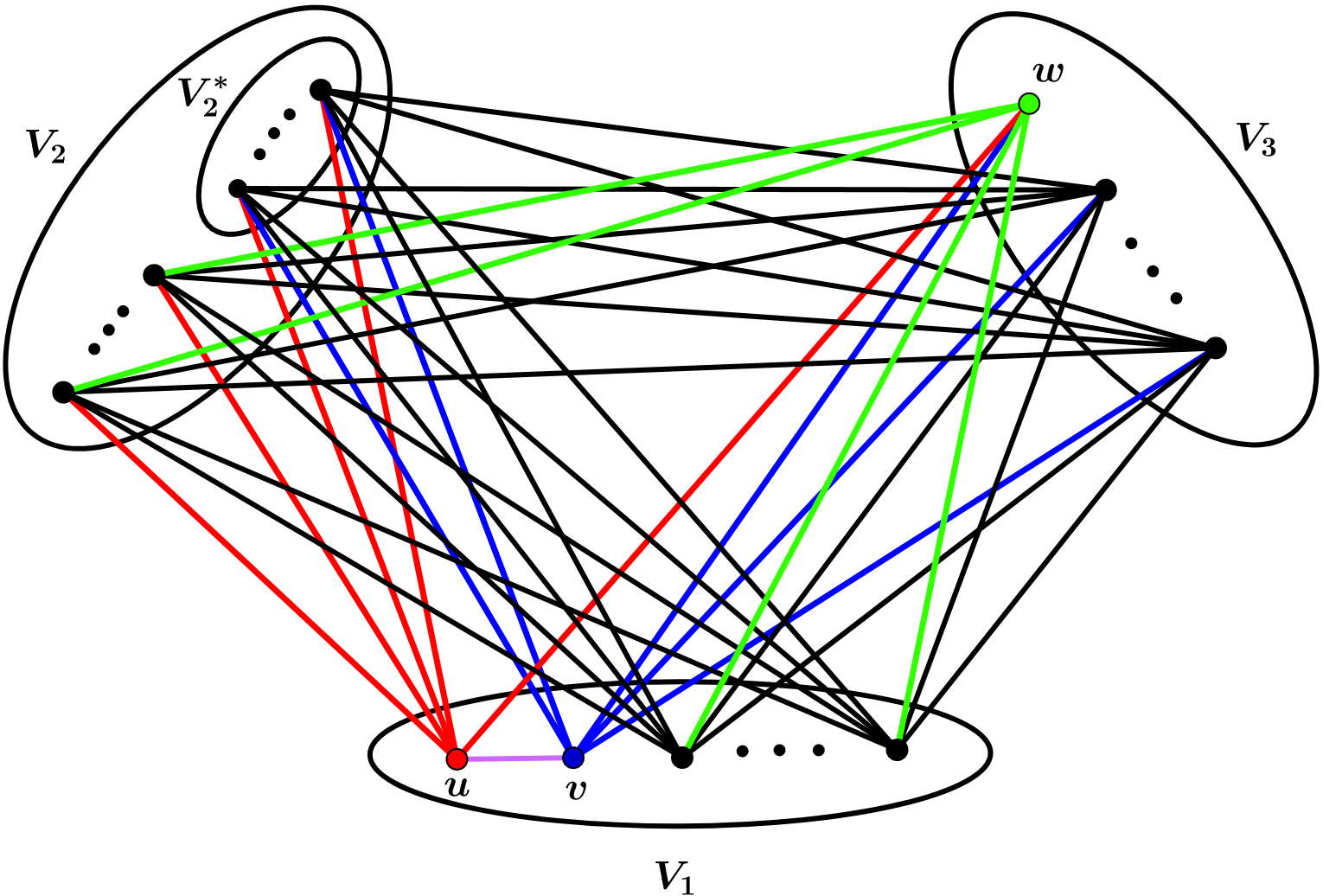}
\caption{$G^*(n,|V^*_2|).$}\label{fig1}
\end{figure}

The main result in this paper is as follows.

\begin{theorem}\label{main1}
Let $k \geq 2$ be an integer and let $n$ be sufficiently large. Then
\begin{eqnarray*}
\mathrm{ex}_4(n, W_{2k+2})=e(T_{n,3})-\left(n-\left\lceil\frac{n}{3}\right\rceil\right)+2=e(T_{n-1,3})+2 \quad\text{and}\quad
\mathrm{Ex}_4(n, W_{2k+2})=\mathscr{G}(n).
\end{eqnarray*}
\end{theorem}

It is interesting to note that the family of extremal graphs $\mathscr{G}(n)$ is independent of $k$. We will elaborate this in the last section. The proof of Theorem \ref{main1} will be given in Section \ref{sec:pf}. The main tools used in our proof are Simonovits' chromatic critical edge theorem (Theorem \ref{th1.7}) and the Erd\H{o}s-Simonovits stability theorem (Lemma \ref{lem2.1}).

\section{Preliminaries}

In this section, we present three preliminary results which are essential to the proof of our main theorem. The first one is the following classical Erd\H{o}s-Simonovits stability theorem.

\begin{lemma}[\!\cite{E1966,E1968,S1968}]\label{lem2.1}
Let $H$ be a graph with $\chi(H)=r+1\geq3.$ For every $\varepsilon>0,$ there exists a constant $\delta>0$ and an integer $n_1(H)$ such that if $G$ is an $H$-free graph on $n\geq n_1(H)$ vertices with $e(G)\geq \big(1-\frac{1}{r}-\delta\big)\frac{n^{2}}{2},$ then $G$ can be obtained from $T_{n,r}$ by adding and deleting at most $\varepsilon n^{2}$ edges.
\end{lemma}

\begin{lemma}[\!\cite{SLMX}]\label{lem2.2}
Let $V_1, \ldots, V_n$ be finite sets. Then
\begin{eqnarray*}
|V_1\cap \cdots\cap V_n|\geq \sum^n_{i=1}|V_i|-(n-1)\left|\bigcup^n_{i=1}V_i\right|.
\end{eqnarray*}
\end{lemma}

By a simple calculation, one can prove the following lemma.

\begin{lemma}\label{lem2.3}
Let $n$ and $r$ be integers with $n\geq r\geq2$. If $G$ is an $r$-partite graph of order $n,$ then $e(G) \leq e(T_{n,r}).$ Furthermore, we have
\begin{eqnarray*}
\left(1-\frac{1}{r}\right)\frac{n^2}{2}-\frac{r}{8}\leq e(T_{n,r})\leq \left(1-\frac{1}{r}\right)\frac{n^2}{2}.
\end{eqnarray*}
\end{lemma}

\section{Proof of Theorem \ref{main1}}
\label{sec:pf}

Throughout this section, we fix an integer $k\geq 2$, assume that $n$ is sufficiently large, and let $G$ be an arbitrary graph in $\mathrm{Ex}_4(n,W_{2k+2})$. It can be easily verified that any graph $H\in \mathscr{G}(n)$ is non-3-partite $W_{2k+2}$-free with size $e(H)= e(T_{n,3})-\big(n-\big\lceil\frac{n}{3}\big\rceil\big)+2=e(T_{n-1,3})+2.$ Since $G \in \mathrm{Ex}_4(n, W_{2k+2})$, it follows that
\begin{eqnarray}\label{eq1}
e(G)\geq e(H)= e(T_{n,3})-\left(n-\left\lceil\frac{n}{3}\right\rceil\right)+2=e(T_{n-1,3})+2.
\end{eqnarray}
Therefore, to prove Theorem \ref{main1} it suffices to prove that $G \in \mathscr{G}(n)$. Recall that $\overline{G}$ denotes the complement of $G$. So, for disjoint vertex sets $A,B\subseteq V(G)$, the number of non-edges of $G$ between $A$ and $B$ is given by $e_{\overline{G}}(A,B)$.

Choose $\varepsilon>0$ to be sufficiently small such that
\begin{eqnarray}\label{eq2}
\max\left\{96k\sqrt{\varepsilon},66n_0(W_{2k+2})\sqrt{\varepsilon}\right\}<1,
\end{eqnarray}
where $n_0(W_{2k+2})$ is guaranteed by Theorem \ref{th1.7}.

We need to prove a series of lemmas before we can prove Theorem \ref{main1}.

\begin{lemma}\label{lem3.1}
For sufficiently large $n,$ we have $e(G)\geq e(T_{n,3})-\varepsilon n^2.$ Moreover, there exists a partition $\{V_1, V_2, V_3\}$ of $V(G)$ with
\begin{eqnarray*}
\left(\frac{1}{3}-2\sqrt{\varepsilon}\right)n\leq |V_i|\leq \left(\frac{1}{3}+2\sqrt{\varepsilon}\right)n\,\ \text{ for }\, i \in \{1,2,3\}
\end{eqnarray*}
such that $\sum^3_{i=1}e(G[V_i])\leq \varepsilon n^2$, $e_{\overline{G}}(V_1,V_2,V_3)\leq \varepsilon n^{2}$, and $e_G(V_1,V_2,V_3)$ attains the maximum among $e_G(U_1,U_2,U_3)$ for all partitions $\{U_1, U_2, U_3\}$ of $V(G)$.
\end{lemma}

\begin{proof}
As $\chi(W_{2k+2})=4,$ it follows from \eqref{eq1} and Lemma \ref{lem2.3} that
\begin{eqnarray*}
e(G)&\geq& e(T_{n,3})-\left(n-\left\lceil\frac{n}{3}\right\rceil\right)+2\\
&\geq& \left(\frac{n^2}{3}-\frac{3}{8}\right)-\left(n-\left\lceil\frac{n}{3}\right\rceil\right)+2\\
&\geq& \left(\left(1-\frac{1}{3}\right)-\frac{4}{3n}+\frac{13}{4n^2}\right)\frac{n^2}{2}.
\end{eqnarray*}
By Lemma \ref{lem2.1}, for any $\varepsilon>0$ and sufficiently large $n,$ $G$ can be obtained from $T_{n,3}$ by adding and deleting at most $\varepsilon n^{2}$ edges, which implies that
\begin{eqnarray}\label{eq3}
e(G)\geq e(T_{n,3})-\varepsilon n^{2}
\end{eqnarray}
and there exists a partition $\{U_1, U_2, U_3\}$ of $V(G)$ with $\big\lfloor\frac{n}{3}\big\rfloor\leq |U_i|\leq \lceil\frac{n}{3}\rceil$ such that $\sum^3_{i=1}e(G[U_i])\leq \varepsilon n^{2}$ and $e_G(U_1,U_2,U_3)\geq e(T_{n,3})-\varepsilon n^{2}.$ Choose a partition $\{V_1, V_2, V_3\}$ of $V(G)$ such that $e_G(V_1,V_2,V_3)$ attains the maximum among $e_G(W_1,W_2,W_3)$ for all partitions $\{W_1, W_2, W_3\}$ of $V(G)$. Then
\begin{eqnarray}\label{eq4}
\sum^3_{i=1}e(G[V_i])\leq \sum^3_{i=1}e(G[U_i])\leq \varepsilon n^{2}.
\end{eqnarray}
Since $e_G(V_1,V_2,V_3)\geq e_G(U_1,U_2,U_3)\geq e(T_{n,3})-\varepsilon n^{2},$ we have $e_{\overline{G}}(V_1,V_2,V_3)\leq\varepsilon n^{2}.$
Let $s=\max\Big\{\big||V_i|-\frac{n}{3}\big|: 1 \le i \le 3\Big\}.$ Without loss of generality, we may assume that $\big||V_1|-\frac{n}{3}\big|=s,$ which implies that $|V_1|=\frac{n}{3}+s$ or $\frac{n}{3}-s.$ It follows from \eqref{eq4} that
\begin{eqnarray}\label{eq5}
e(G) &=& \sum_{1\leq i<j\leq 3}e_G(V_i,V_j)+\sum^3_{i=1}e(G[V_i])\nonumber\\
&\leq& |V_1|(n-|V_1|)+|V_2||V_3|+\varepsilon n^{2}\nonumber\\
&=& |V_1|(n-|V_1|)+\frac{1}{2}\big((|V_2|+|V_3|)^2-(|V_2|^2+|V_3|^2)\big)+\varepsilon n^{2}\nonumber\\
&\leq& |V_1|(n-|V_1|)+\frac{1}{2}\big((|V_2|+|V_3|)^2-\frac{1}{2}(|V_2|+|V_3|)^2\big)+\varepsilon n^{2}\nonumber\\
&=& |V_1|(n-|V_1|)+\frac{1}{4}(n-|V_1|)^2+\varepsilon n^{2}\nonumber\\
&=& \frac{n^2}{3}-\frac{3}{4}s^2+\varepsilon n^{2}.
\end{eqnarray}
On the other hand, by \eqref{eq3} and Lemma \ref{lem2.3}, we have
\begin{eqnarray}\label{eq6}
e(G)&\geq& e(T_{n,3})-\varepsilon n^{2}\geq \frac{n^2}{3}-\varepsilon n^{2}-\frac{3}{8}\geq \frac{n^2}{3}-2\varepsilon n^{2}
\end{eqnarray}
for sufficiently large $n.$ Combining \eqref{eq5} and \eqref{eq6}, we obtain that $s\leq 2\sqrt{\varepsilon}n.$ Hence
\begin{eqnarray*}
\left(\frac{1}{3}-2\sqrt{\varepsilon}\right)n\leq |V_i|\leq \left(\frac{1}{3}+2\sqrt{\varepsilon}\right)n
\end{eqnarray*}
for $i \in \{1,2,3\}$.
\end{proof}

\begin{lemma}\label{lem3.2}
For $i \in \{1,2,3\}$ and $x\in V_i$, we have $d_G(x,V_i)<2\sqrt{\varepsilon}n.$
\end{lemma}
\begin{proof}
Without loss of generality, suppose to the contrary that there exists a vertex $u_0 \in V_1$ such that $d_G(u_0,V_1) \geq 2\sqrt{\varepsilon}n.$ Recall that the partition $\{V_1, V_2, V_3\}$ of $V(G)$ maximizes $e_G(U_1,U_2,U_3)$ among all partitions $\{U_1, U_2, U_3\}$ of $V(G)$. Then $d_G(u_0,V_i) \geq d_G(u_0,V_1)\geq 2\sqrt{\varepsilon}n$ for each $i\in \{2,3\}.$ Let $N_i\subseteq N_G(u_0,V_i)$ be a subset of size $|N_i|=\big\lceil 2\sqrt{\varepsilon}n \big\rceil$ for each $i\in \{1,2,3\}$ and let $F=G[N_1\cup N_2\cup N_3].$
Then $|V(F)|= 3\lceil 2\sqrt{\varepsilon}n\rceil \geq 4k-2$ for large enough $n.$ By Lemma \ref{lem3.1}, $e_{\overline{G}}(V_1,V_2,V_3)\leq \varepsilon n^{2}.$ Hence
\begin{eqnarray}\label{eq7}
e(F) \geq 3\big\lceil 2\sqrt{\varepsilon}n\big\rceil^2-\varepsilon n^2 \geq 11\varepsilon n^2>\frac{|V(F)|^2}{4}\geq e(T_{|V(F)|,2})
\end{eqnarray}
for sufficiently large $n.$ Note that $C_{2k+1}$ is color-critical with $\chi(C_{2k+1})=3.$ Combining \eqref{eq7} and Theorem \ref{th1.3}, we obtain that $F$ must contain a copy of $C_{2k+1}.$ Since $V(F) \subseteq N_G(u_0),$ $u_0$ and this odd cycle form a copy of $W_{2k+2}$ in $G,$ which contradicts the fact that $G$ is $W_{2k+2}$-free.
\end{proof}

Define
$$
L=\left\{v\in V(G): d_G(v)\leq \left(\frac{2}{3}-5\sqrt{\varepsilon}\right)n\right\}.
$$
The next lemma gives an upper bound for the size of $L.$

\begin{lemma}\label{lem3.3}
$|L|< \sqrt{\varepsilon}n.$
\end{lemma}

\begin{proof}
Suppose to the contrary that $|L|\geq\sqrt{\varepsilon}n.$ Then there exists a subset $L'\subseteq L$ with $|L'|=\lfloor\sqrt{\varepsilon}n\rfloor.$ Let $n' = n-|L'|.$
By \eqref{eq3} and Lemma \ref{lem2.3}, we have
\begin{eqnarray*}
e(G[V(G)\setminus L'])&\geq& e(G)-\sum_{v\in L'}d_G(v)\\
&\geq& e(T_{n,3})-\varepsilon n^{2}-\sqrt{\varepsilon}n^2 \left(\frac{2}{3}-5\sqrt{\varepsilon}\right)\\
&\geq& \frac{n^2}{3}+4\varepsilon n^{2}-\frac{2}{3}\sqrt{\varepsilon}n^2-\frac{3}{8}\\
&>& \frac{1}{3}\Big(n-\Big\lfloor\sqrt{\varepsilon}n\Big\rfloor\Big)^2 \geq e(T_{n',3})
\end{eqnarray*}
for sufficiently large $n.$ Note that $W_{2k+2}$ is color-critical and $\chi(W_{2k+2})=4.$ By Theorem \ref{th1.7}, there exists an integer $n_0(W_{2k+2})$ such that $\mathrm{ex}(n', W_{2k+2})=e(T_{n',3})$ for $n'\geq n_0(W_{2k+2}).$ Since $e(G[V(G)\setminus L'])>e(T_{n',3})$, $G[V(G)\setminus L']$ must contain a copy of $W_{2k+2},$ a contradiction.
\end{proof}

\begin{lemma}\label{lem3.4}
For $i \in \{1,2,3\}$, we have $e(G[V_i\setminus L])=0.$
\end{lemma}

\begin{proof}
Without loss of generality, suppose to the contrary that $e(G[V_1\setminus L])\geq 1.$ Let $u_1u_2$ be an edge of $G[V_1\setminus L]$ and let $n_2$ be a fixed integer such that $n_2=\max\left\{k+1, \Big\lceil\frac{n_0(W_{2k+2})}{3}\Big\rceil\right\}.$ By Lemmas \ref{lem3.1}, \ref{lem3.3} and \eqref{eq2}, we have
\begin{eqnarray*}
|V_1\setminus L|\geq |V_1| - |L|\geq \left(\frac{1}{3}-3\sqrt{\varepsilon}\right)n>n_2
\end{eqnarray*}
for sufficiently large $n.$ So we can choose $n_2-2$ vertices $u_3, u_4, \ldots, u_{n_2}$ from $V_1 \setminus (L \cup \{u_1, u_2\}).$ Since none of the vertices $u_1, \ldots, u_{n_2}$ belong to $L,$ we have $d_G(u_i)>\big(\frac{2}{3}-5\sqrt{\varepsilon}\big)n$ for any $i\in \{1,2,\ldots,n_2\}.$ By Lemma \ref{lem3.2}, $d_G(u_i,V_1)<2\sqrt{\varepsilon}n.$
It follows from Lemma \ref{lem3.1} and \eqref{eq2} that
\begin{eqnarray}\label{eq8}
d_G(u_i,V_j)&\geq& d_G(u_i)-d_G(u_i,V_1)-(n-|V_1|-|V_j|)\nonumber\\
&>& \left(\frac{2}{3}-5\sqrt{\varepsilon}\right)n-2\sqrt{\varepsilon}n-\left(\frac{1}{3}+2\sqrt{\varepsilon}\right)n\nonumber\\
&=& \left(\frac{1}{3}-9\sqrt{\varepsilon}\right)n,
\end{eqnarray}
where $j\in \{2,3\}$ and $i\in \{1,2,\ldots,n_2\}.$ Next we consider the common neighbors of $u_1, \ldots, u_{n_2}$ in $V_2 \setminus L.$
Combining Lemmas \ref{lem2.2}, \ref{lem3.1} and \ref{lem3.3}, \eqref{eq8} and \eqref{eq2}, we have
\begin{eqnarray*}
\Bigg|\Big(\bigcap_{i\in\{1,\ldots,n_2\}}N_G(u_i,V_2)\Big)\setminus L\Bigg|&\geq& \sum^{n_2}_{i=1}d_G(u_i,V_2)-(n_2-1)|V_2|-|L|\\
&>& n_2\left(\frac{1}{3}-9\sqrt{\varepsilon}\right)n-(n_2-1)\left(\frac{1}{3}+2\sqrt{\varepsilon}\right)n-\sqrt{\varepsilon}n\\
&=& \left(\frac{1}{3}-11\sqrt{\varepsilon}n_2+\sqrt{\varepsilon}\right)n>n_2
\end{eqnarray*}
for sufficiently large $n.$ Hence there exist $n_2$ vertices $u'_1, u'_2, \ldots, u'_{n_2} \in \big(\bigcap_{i\in\{1,\ldots,n_2\}}N_G(u_i,V_2)\big)\setminus L.$ Based on the technique above, we have
\begin{eqnarray}\label{eq9}
d_G(u'_i,V_3)\geq d_G(u'_i)-d_G(u'_i,V_2)-|V_1|> \left(\frac{1}{3}-9\sqrt{\varepsilon}\right)n
\end{eqnarray}
for each $i\in \{1,2,\ldots,n_2\}.$ Combining Lemmas \ref{lem2.2}, \ref{lem3.1} and \ref{lem3.3}, \eqref{eq8}, \eqref{eq9} and \eqref{eq2}, we have
\begin{eqnarray*}
& &\left|\left(\bigcap_{i\in\{1,\ldots,n_2\}}N_G(u_i,V_3)\right)\cap \left(\bigcap_{i\in\{1,\ldots,n_2\}}N_G(u'_i,V_3)\right)\right|\\
&\geq& \sum^{n_2}_{i=1}d_G(u_i,V_3)+\sum^{n_2}_{i=1}d_G(u'_i,V_3)-(2n_2-1)|V_3|\\
&>& 2n_2\left(\frac{1}{3}-9\sqrt{\varepsilon}\right)n-(2n_2-1)\left(\frac{1}{3}+2\sqrt{\varepsilon}\right)n\\
&=& \left(\frac{1}{3}-22\sqrt{\varepsilon}n_2+2\sqrt{\varepsilon}\right)n>n_2
\end{eqnarray*}
for sufficiently large $n.$ So there exist $n_2$ vertices $u''_1,u''_2,\ldots, u''_{n_2} \in \big(\bigcap_{i\in\{1,\ldots,n_2\}}N_G(u_i,V_3)\big)\cap \big(\bigcap_{i\in\{1,\ldots,n_2\}}N_G(u'_i,V_3)\big),$ and hence $\{u_1,\ldots,u_{n_2},u'_1,\ldots,u'_{n_2},u''_1,\ldots,u''_{n_2}\}$ induces a subgraph which contains a copy of $T'_{3n_2,3},$ where $T'_{3n_2,3}$ denotes the graph obtained by adding exactly one edge to $T_{3n_2,3}.$ By Theorem \ref{th1.7}, $T'_{3n_2,3}$ contains a copy of $W_{2k+2}.$ This implies that $G$ contains a copy of $W_{2k+2},$ a contradiction.
\end{proof}

\begin{lemma}\label{lem3.5}
For any $w\in L$ and $a \in \{1,2,3\}$, let $G'=G-\{wx:x\in N_G(w)\}+\{wx:x\in V(G)\setminus(V_a\cup L)\}.$ Then the following hold:
\begin{enumerate}[\rm (i)]
	\item $G'$ is a $3$-partite $W_{2k+2}$-free graph;
	\item for any $1 \le i \le 3$ and $u_1,u_2\in V_i\setminus L,$ $u_1$ and $u_2$ belong to the same part of $G',$ that is, there exists a partition $\{V'_1, V'_2, V'_3\}$ of $V(G')$ such that $\sum^3_{i=1}e(G'[V'_i])=0$ and $V_i\setminus L\subseteq V'_i.$
\end{enumerate}
\end{lemma}

\begin{proof}
(i) Without loss of generality, we may assume that $a=1.$ We first prove that $G'$ is $W_{2k+2}$-free. Suppose to the contrary that $G'$ contains a copy of $W_{2k+2}.$ If $w$ is not a vertex of this $W_{2k+2},$ then there exists a copy of $W_{2k+2}$ in $G[V(G)\setminus \{w\}].$ Hence $w$ is a vertex of this $W_{2k+2}.$ By Lemma \ref{lem3.4}, we have $e(G[V_i\setminus L])=0.$ Then $G'[N_{G'}(w)]$ is a bipartite graph, which implies that it contains no $C_{2k+1}.$ Therefore, $w$ is not the center vertex of this $W_{2k+2},$ and $w$ lies on the rim $C$ of this $W_{2k+2}.$ Note that $N_{G'}(w)\subseteq (V_2\cup V_3)\setminus L.$ By symmetry, suppose that $v_0\in V_2\setminus L$ is the center vertex of this $W_{2k+2}.$
It follows from Lemma \ref{lem3.4} that $N_{G'}(w)\cap N_{G'}(v_0)\subseteq V_3 \setminus L.$ Let $w_1$ and $w_2$ be the two vertices adjacent to $w$ on $C.$ Next we show that $w_1,$ $w_2$ and $v_0$ share a common neighbor in $V_1 \setminus (L \cup V(C)).$ Note that $w_1, w_2, v_0 \notin L.$ Then $d_G(v_0)>\big(\frac{2}{3}-5\sqrt{\varepsilon}\big)n$ and $d_G(w_i)>\big(\frac{2}{3}-5\sqrt{\varepsilon}\big)n$ for each $i\in \{1,2\}.$ According to Lemma \ref{lem3.2}, we have $d_G(v_0,V_2)<2\sqrt{\varepsilon}n$ and $d_G(w_i,V_3)<2\sqrt{\varepsilon}n$ for each $i\in \{1,2\}.$ It follows from Lemma \ref{lem3.1} that
\begin{eqnarray}\label{eq10}
d_G(v_0,V_1)\geq d_G(v_0)-d_G(v_0,V_2)-|V_3|>\left(\frac{1}{3}-9\sqrt{\varepsilon}\right)n.
\end{eqnarray}
Similarly, we have
\begin{eqnarray}\label{eq11}
d_G(w_i,V_1)\geq \left(\frac{1}{3}-9\sqrt{\varepsilon}\right)n
\end{eqnarray}
for each $i\in \{1,2\}.$ Combining Lemmas \ref{lem2.2}, \ref{lem3.1} and \ref{lem3.3}, \eqref{eq10}, \eqref{eq11} and \eqref{eq2}, we have
\begin{eqnarray*}
& &\big|(N_G(v_0,V_1)\cap N_G(w_1,V_1)\cap N_G(w_2,V_1))\setminus(L\cup V(C))\big|\\
&\geq& d_G(v_0,V_1)+d_G(w_1,V_1)+d_G(w_2,V_1)-2|V_1|-|L|-(2k+1)\\
&>& 3\left(\frac{1}{3}-9\sqrt{\varepsilon}\right)n-2\left(\frac{1}{3}+2\sqrt{\varepsilon}\right)n-\sqrt{\varepsilon}n-(2k+1)\\
&=& \left(\frac{1}{3}-32\sqrt{\varepsilon}\right)n-2k-1>0
\end{eqnarray*}
for large enough $n.$ Let $w'\in (N_G(v_0,V_1)\cap N_G(w_1,V_1)\cap N_G(w_2,V_1))\setminus (L\cup V(C)).$ By exchanging $w$ with $w',$ we obtain a new $W_{2k+2}$ in $G'[V(G') \setminus \{w\}].$ Note that $G'[V(G') \setminus \{w\}]= G[V(G) \setminus \{w\}].$ Then $G$ contains a copy of $W_{2k+2},$ a contradiction.
Hence $G'$ is $W_{2k+2}$-free.

Next we show that $G'$ is a $3$-partite graph. Suppose to the contrary that $G'$ is a non-$3$-partite $W_{2k+2}$-free graph. Combining Lemmas \ref{lem3.1} and \ref{lem3.3} and $w\in L,$ we have
\begin{eqnarray*}
e(G') &=& e(G)-d_G(w)+|V(G)\setminus (V_1\cup L)|\\
&\geq& e(G)-\left(\frac{2}{3}-5\sqrt{\varepsilon}\right)n+n-\left(\frac{1}{3}+2\sqrt{\varepsilon}\right)n-\sqrt{\varepsilon}n\\
&=& e(G)+2\sqrt{\varepsilon}n>e(G),
\end{eqnarray*}
which contradicts the maximality of $e(G).$

(ii) Without loss of generality, we may assume to the contrary that $u_1,u_2\in V_1\setminus L$ in $G,$ but $u_1\in V'_1$ and $u_2\in V'_2$ in $G'.$ According to the definition of $L,$ we have
\begin{eqnarray}\label{eq12}
d_G(u_i)>\left(\frac{2}{3}-5\sqrt{\varepsilon}\right)n
\end{eqnarray}
for $i\in \{1,2\}.$ By Lemmas \ref{lem3.4} and \ref{lem3.3}, we have
\begin{eqnarray}\label{eq13}
d_G(u_i,V_1)\leq |V_1\cap L|\leq \sqrt{\varepsilon}n.
\end{eqnarray}
Combining \eqref{eq12} and \eqref{eq13} and Lemma \ref{lem3.1}, we deduce that
\begin{eqnarray}\label{eq14}
d_G(u_i,V_j) &\geq& d_G(u_i)-d_G(u_i,V_1)-(n-|V_1|-|V_j|)>\left(\frac{1}{3}-8\sqrt{\varepsilon}\right)n
\end{eqnarray}
for $i\in \{1,2\}$ and $j\in \{2,3\}.$ Next we consider the common neighbors of $u_1$ and $u_2$ in $V_2\setminus L.$ It follows from Lemmas \ref{lem2.2} and \ref{lem3.3}, \eqref{eq14} and \eqref{eq2} that
\begin{eqnarray*}
|(N_G(u_1,V_2)\cap N_G(u_2,V_2))\setminus L| &\geq& \sum^2_{i=1} d_G(u_i,V_2)-|V_2|-|L|\\
&>& 2\left(\frac{1}{3}-8\sqrt{\varepsilon}\right)n-\left(\frac{1}{3}+2\sqrt{\varepsilon}\right)n-\sqrt{\varepsilon}n\\
&=& \left(\frac{1}{3}-19\sqrt{\varepsilon}\right)n>0.
\end{eqnarray*}
Let $z\in (N_G(u_1,V_2)\cap N_G(u_2,V_2))\setminus L.$ Similarly to \eqref{eq14}, we have $d_G(z,V_3)>\big(\frac{1}{3}-8\sqrt{\varepsilon}\big)n.$ Combining Lemmas \ref{lem2.2} and \ref{lem3.3}, \eqref{eq14} and \eqref{eq2}, we have
\begin{eqnarray*}
|(N_G(z,V_3)\cap N_G(u_1,V_3)\cap N_G(u_2,V_3))\setminus L| &\geq& \sum^2_{i=1} d_G(u_i,V_3)+d_G(z,V_3)-2|V_3|-|L|\\
&>& \left(\frac{1}{3}-29\sqrt{\varepsilon}\right)n>0.
\end{eqnarray*}
Let $z'\in (N_G(z,V_3)\cap N_G(u_1,V_3)\cap N_G(u_2,V_3))\setminus L.$ Note that the modified vertex $w \in L.$ Then $w\notin \{u_1, u_2, z, z'\}.$ By the construction of $G',$ all edges among $u_1, u_2, z, z'$ in $G$ are preserved in $G'.$ Since $u_1\in V'_1$ and $u_2\in V'_2,$ we have $\{z,z'\}\subset V'_3.$ However, $zz'\in E(G'),$ which contradicts the fact that  $\sum^3_{i=1}e(G'[V'_i])=0.$
\end{proof}

\begin{lemma}\label{lem3.6}
For any $S\subseteq V(G),$ we have
\begin{eqnarray*}
e(G)&\leq& e(K_{|V_1|,|V_2|,|V_3|})-\sum^3_{i=1}\sum_{x\in V_i\cap S}(n-|V_i|)+e(K_{|V_1|,|V_2|,|V_3|}[S])+\sum_{x\in S}d_G(x)\\
& & -e(G[S])+|S\setminus L||L\setminus S|.
\end{eqnarray*}
\end{lemma}
\begin{proof}
Combining Lemmas \ref{lem3.1} and \ref{lem3.3}, for any $x\in V_i\cap L,$ we have
\begin{eqnarray}\label{eq15}
d_G(x)\leq \left(\frac23-5\sqrt{\varepsilon}\right)n< |V(G)\setminus (V_i\cup L)|.
\end{eqnarray}
We construct $G_1$ from $G$ by modifying the edges incident to vertices in $L\setminus S$ as follows.
\begin{eqnarray*}
G_1
&=&G-\bigcup_{i=1}^3\bigcup_{x\in (V_i\cap L)\setminus S}
\{xy:y\in N_G(x)\}+\bigcup_{i=1}^3\bigcup_{x\in (V_i\cap L)\setminus S}
\{xy:y\in V(G)\setminus (V_i\cup L)\}.
\end{eqnarray*}
By \eqref{eq15}, $e(G)\leq e(G_1).$ Moreover, $V_i\setminus S$ is independent in $G_1$ for each $i \in \{1,2,3\}$ and $\big(E(G_1[S])\cup E_{G_1}(S,V(G_1)\setminus S)\big)\subseteq \big(E(G[S])\cup E_G(S,V(G)\setminus S)\cup E(K_{S\setminus L, L\setminus S})\big),$ where $K_{S\setminus L, L\setminus S}$ is a complete bipartite graph with bipartition $\{S\setminus L, L\setminus S\}.$ Hence
\begin{eqnarray*}
e(G) &\leq& e(G_1)= e(G_1[V(G)\setminus S])+e(G_1[S])+e_{G_1}(S,V(G)\setminus S)\nonumber\\
&\leq& e(K_{|V_1\setminus S|,|V_2\setminus S|,|V_3\setminus S|})+e(G[S])+e_G(S,V(G)\setminus S)+(|S\setminus L||L\setminus S|)\nonumber\\
&=& e(K_{|V_1|,|V_2|,|V_3|})-\sum^3_{i=1}\sum_{x\in V_i\cap S}(n-|V_i|)+e(K_{|V_1|,|V_2|,|V_3|}[S])
+\sum_{x\in S}d_G(x)\nonumber\\
& & -e(G[S])+|S\setminus L||L\setminus S|.
\end{eqnarray*}
This completes the proof.
\end{proof}

By Lemma \ref{lem3.5}, we deduce the following lemma.

\begin{lemma}\label{lem3.7}
Let $a,b\in \{1,2,3\}$ with $a\neq b.$ Suppose that $uv\in E(G[V_a])$ and $w\in L\setminus\{u,v\}.$ For any $j\in \{1,2,3\},$ let
$$G'=G-\{wx:x\in N_G(w)\}+\{wx:x\in V(G)\setminus(V_j\cup L)\}.$$
Then there exists a partition $\{V'_1, V'_2, V'_3\}$ of $V(G')$ such that $\sum^3_{i=1}e(G'[V'_i])=0$ and $V_i\setminus L\subseteq V'_i$ for each $i\in \{1,2,3\}.$ Moreover, if $u\in V'_b,$ then $N_G(u,V_b)=\{w\}$ and $N_G(u,V_a)=\{v\}.$
\end{lemma}
\begin{proof}
Combining $V_i\setminus L\subseteq V'_i,$ $u\in V_a$ in $G$ and $u\in V'_b$ in $G',$ we have $u\in L.$
First we prove that $N_G(u,V_b\setminus\{w\})=\emptyset.$
Suppose to the contrary that there exists a vertex $w'\in V_b\setminus\{w\}$ such that $uw'\in E(G).$ As $u\in V'_b$ and $e(G'[V'_b])=0,$ we have $w'\notin V'_b.$ Then $w'\in L.$ Let $w'\in V'_c,$ where $c\neq b.$
Since $e(G'[V'_b])=0$ and $V_b\setminus L\subseteq V'_b,$ we have $d_G(u,V_b\setminus L)=0.$
By the maximality of $e_G(V_1,V_2,V_3),$ we have $d_G(u,V_a)\leq d_G(u,V_b)\leq |V_b\cap L|.$ Then
\begin{eqnarray}\label{eq16}
d_G(u)-(n-|V_a|)\leq -|V_b|+2|V_b\cap L|.
\end{eqnarray}
Using the technique leading to \eqref{eq16}, we deduce that
\begin{eqnarray}\label{eq17}
d_G(w')-(n-|V_b|)\leq -|V_c|+2|V_c\cap L|.
\end{eqnarray}
Moreover, if $w\in V_t,$ then by Lemmas \ref{lem3.1} and \ref{lem3.3}, we have
\begin{eqnarray}\label{eq18}
d_G(w)-(n-|V_t|)\leq \left(\frac23-5\sqrt{\varepsilon}\right)n-\left(\frac23-2\sqrt{\varepsilon}\right)n =-3\sqrt{\varepsilon}n.
\end{eqnarray}
Let $S=\{u,w,w'\}.$ Then $S\subseteq L.$ Combining Lemma \ref{lem3.6}, \eqref{eq16}, \eqref{eq17} and \eqref{eq18}, we have
\begin{eqnarray*}
e(G)&\leq&e(K_{|V_1|,|V_2|,|V_3|})-\sum^3_{i=1}\sum_{x\in V_i\cap S}(n-|V_i|)+e(K_{|V_1|,|V_2|,|V_3|}[S])
+\sum_{x\in S}d_G(x)-e(G[S])\\
&\leq& e(K_{|V_1|,|V_2|,|V_3|})-|V_b|-|V_c|+2|V_b\cap L|+2|V_c\cap L|-3\sqrt{\varepsilon}n+3\\
&\leq& e(T_{n-1,3})+2|L|-3\sqrt{\varepsilon}n+3\\
&\leq& e(T_{n-1,3})-\sqrt{\varepsilon}n+3<e(T_{n-1,3})+2
\end{eqnarray*}
for sufficiently large $n,$ contradicting \eqref{eq1}. Hence $N_G(u,V_b\setminus\{w\})=\emptyset.$ Recall that $uv\in E(G[V_a]).$ If $N_G(u,V_b)=\emptyset,$ then $d_G(u,V_a)>d_G(u,V_b),$ which contradicts the maximality of $e_G(V_1,V_2,V_3).$
Hence $N_G(u,V_b)\neq\emptyset,$ which implies that $w\in V_b\cap L$ and $N_G(u,V_b)=\{w\}.$ Thus $d_G(u,V_b)=1.$ By the maximality of $e_G(V_1,V_2,V_3),$ $d_G(u,V_a)\leq d_G(u,V_b)=1.$ Since $uv\in E(G[V_a]),$ we have $N_G(u,V_a)=\{v\}.$
\end{proof}

\begin{lemma}\label{lem3.8}
For any $i,j\in \{1,2,3\}$ with $i\neq j,$ if there exists an edge $uv\in E(G[V_i]),$ then $(V_i\cap L)\setminus \{u,v\}=\emptyset.$ In particular, if $|\{u,v\}\cap L|=1,$ then $|V_j\cap L|\leq 1.$
\end{lemma}
\begin{proof}
We first prove that $(V_i\cap L)\setminus \{u,v\}=\emptyset.$ Without loss of generality, suppose to the contrary that there exists a vertex $w\in (V_1\cap L)\setminus\{u,v\}.$
Let $$G'=G-\{wx:x\in N_G(w)\}+\{wx:x\in V(G)\setminus(V_1\cup L)\}.$$ By Lemma \ref{lem3.5}, $G'$ is a $3$-partite graph and there exists a partition $\{V'_1, V'_2, V'_3\}$ of $V(G')$ such that $\sum^3_{i=1}e(G'[V'_i])=0$ and $V_i\setminus L\subseteq V'_i$ for each $i\in \{1,2,3\}.$ Then we have $u\notin V'_1$ or $v\notin V'_1.$ By symmetry, we suppose that $u\in V'_2.$ It follows from Lemma \ref{lem3.7} that $N_G(u,V_2)=\{w\},$ which contradicts $w\in V_1.$
Next we show that if $|\{u,v\}\cap L|=1,$ then $|V_j\cap L|\leq 1.$ Without loss of generality, we suppose that $u\in L,$ $v\notin L$ and there exist two vertices $w_1,w_2\in V_j\cap L.$
By the technique above, we have $N_G(u,V_j)=\{w_1\}$ and $N_G(u,V_j)=\{w_2\},$ a contradiction.
\end{proof}

\begin{lemma}\label{lem3.9}
$1\leq |L|\leq 3.$
\end{lemma}
\begin{proof}
If $|L|=0,$ then by Lemma \ref{lem3.4}, we have $\sum^3_{i=1}e(G[V_i])=0.$ Hence $G$ is $3$-partite, a contradiction. Therefore, $|L|\geq 1.$ Suppose to the contrary that $|L|\geq 4.$ Since $G$ is a non-$3$-partite graph, $\sum^3_{i=1}e(G[V_i])\geq 1.$ Without loss of generality, we may assume that $uv\in E(G[V_1]).$ Since $|L|\geq 4,$ we can select a vertex $w_1\in L\setminus\{u,v\}$ and construct a new graph $$G'=G-\{w_1x:x\in N_G(w_1)\}+\{w_1x:x\in V(G)\setminus(V_1\cup L)\}.$$
By Lemma \ref{lem3.5}, $G'$ is a $3$-partite $W_{2k+2}$-free graph and there exists a partition $\{V'_1, V'_2, V'_3\}$ of $V(G')$ such that $\sum^3_{i=1}e(G'[V'_i])=0$ and $V_i\setminus L\subseteq V'_i$ for each $i\in \{1,2,3\}.$ Note that $uv\in E(G').$ Then $u\notin V'_1$ or $v\notin V'_1.$ By symmetry, we may assume that $u\in V'_2$ in $G'.$ Since $V_i\setminus L\subseteq V'_i,$ we have $u\in L.$ It follows from Lemma \ref{lem3.7} that
\begin{eqnarray}\label{eq19}
N_G(u,V_2)=\{w_1\}
\end{eqnarray}
and $N_G(u,V_1)=\{v\}.$
Since $|L|\geq 4,$ we can take a vertex $w_2\in L\setminus \{u,v,w_1\}.$ Let $$G''=G-\{w_2x:x\in N_G(w_2)\}+\{w_2x:x\in V(G)\setminus(V_1\cup L)\}.$$
By Lemma \ref{lem3.5}, $G''$ is a 3-partite $W_{2k+2}$-free graph and there exists a partition $\{V''_1, V''_2, V''_3\}$ of $V(G'')$ such that $\sum^3_{i=1}e(G''[V''_i])=0$ and $V_i\setminus L\subseteq V''_i.$
Hence $u\notin V''_1$ or $v\notin V''_1$ in $G''.$
Recall that $N_G(u,V_2)=\{w_1\}$ and $N_G(u,V_1)=\{v\}.$
If $u\in V''_2,$ then by Lemma \ref{lem3.7}, we have $N_G(u,V_2)=\{w_2\},$ contradicting \eqref{eq19}. If $u\in V''_3,$ then $N_G(u,V_3)=\{w_2\}.$ Hence $d_G(u)=3.$ Combining $\{u,w_1\}\subseteq L,$ Lemmas \ref{lem3.1}, \ref{lem3.6} and $uw_1\in E(G),$ we have
\begin{eqnarray*}
e(G)&\leq& e(K_{|V_1|,|V_2|,|V_3|})-\sum^3_{i=1}\sum_{x\in V_i\cap \{u,w_1\}}(n-|V_i|)+1+\sum_{x\in \{u,w_1\}}d_G(x)-1\\
&\leq& e(K_{|V_1|,|V_2|,|V_3|})-(|V_2|+|V_3|)-(|V_1|+|V_3|)+d_G(u)+d_G(w_1)\\
&\leq& e(T_{n-1,3})-3\sqrt{\varepsilon}n+3<e(T_{n-1,3})+2
\end{eqnarray*}
for large enough $n,$ contradicting \eqref{eq1}.
If $v\notin V''_1$ in $G'',$ then by Lemma \ref{lem3.7}, there exists $t\in \{2,3\}$ such that $N_G(v,V_t)=\{w_2\}$ and $N_G(v,V_1)=\{u\}.$ Hence $d_G(v)\leq (|V_2|+|V_3|)-|V_t|+2.$ Since $V_i\setminus L\subseteq V''_i,$ we have $v\in L.$
Note that $d_G(u)\leq |V_3|+2.$ Combining $\{u,v,w_1\}\subseteq L$ and Lemmas \ref{lem3.1} and \ref{lem3.6}, we have
\begin{eqnarray*}
e(G)&\leq& e(K_{|V_1|,|V_2|,|V_3|})-2(|V_2|+|V_3|)-(|V_1|+|V_3|)+d_G(u)+d_G(v)+d_G(w_1)+3\\
&\leq& e(K_{|V_1|-1,|V_2|,|V_3|})-(|V_1|+|V_t|)+d_G(w_1)+7\\
&\leq& e(T_{n-1,3})-3\sqrt{\varepsilon}n+7<e(T_{n-1,3})+2
\end{eqnarray*}
for large enough $n,$ a contradiction.
\end{proof}

Without loss of generality, we suppose that $|V_1\cap L|\geq |V_2\cap L|\geq |V_3\cap L|.$ By Lemma \ref{lem3.9}, $1\leq |L|\leq 3.$ Hence $|V_3\cap L|\leq|V_2\cap L|\leq 1.$

\begin{lemma}\label{lem3.10}
$|V_3\cap L|=0$ and $e(G[V_3])=0.$
\end{lemma}
\begin{proof}
Suppose to the contrary that $|V_3\cap L|=1.$ Then $|V_1\cap L|=|V_2\cap L|=|V_3\cap L|=1.$
Since $G$ is non-$3$-partite, we have $\sum^3_{i=1}e(G[V_i])\geq 1.$
By Lemma \ref{lem3.4} and $|V_1\cap L|=1,$ without loss of generality, suppose that $uv\in E(G[V_1])$ and $u\in L.$
Let $$G_j=G-\{w_jx:x\in N_G(w_j)\}+\{w_jx:x\in V(G)\setminus(V_1\cup L)\},$$ where $\{w_j\}=V_{j+1}\cap L$ with $j\in \{1,2\}.$
By Lemma \ref{lem3.5}, $G_1$ and $G_2$ are 3-partite $W_{2k+2}$-free graphs and there exists a partition $\{V^j_1, V^j_2, V^j_3\}$ of $V(G_j)$ such that $\sum^3_{i=1}e(G_j[V^j_i]) = 0$ and $V_i \setminus L\subseteq V^j_i,$ where $i\in \{1,2,3\}$ and $j\in \{1,2\}.$ Then $u\notin V^j_1.$ By Lemma \ref{lem3.7}, we have $N_G(u,V_1)=\{v\},$ $N_G(u,V_2)=\{w_1\}$ and $N_G(u,V_3)=\{w_2\},$ which implies that $d_G(u)=3.$
Combining Lemmas \ref{lem3.1} and \ref{lem3.6}, we have
\begin{eqnarray*}
e(G) &\leq& e(K_{|V_1|,|V_2|,|V_3|})-(n-|V_1|)-(n-|V_2|)+d_G(u)+d_G(w_1)+1\\
&\leq& e(K_{|V_1|-1,|V_2|,|V_3|})-3\sqrt{\varepsilon}n+4< e(T_{n-1,3})+2
\end{eqnarray*}
for sufficiently large $n,$ which contradicts \eqref{eq1}. Hence $|V_3\cap L|=0.$ It follows from Lemma \ref{lem3.4} that $e(G[V_3])=0.$
\end{proof}

\begin{lemma}\label{lem3.11}
Let $a,b,c\in \{1,2,3\}$ be distinct integers, $w\in V_a\setminus L$ and $u_1,v_1\in N_G(w)\setminus L,$ where $u_1\neq v_1.$ Let $S\subseteq V(G)$ with $|S|\leq 2k+2.$ Then the following hold:
\begin{enumerate}[\rm (i)]
	\item if $u_1\in (V_b\cap N_G(w))\setminus (L\cup S)$ and $v_1\in (V_c\cap N_G(w))\setminus (L\cup S),$ then for any $\ell \in \{1,2,\ldots,k\},$ there exists an odd path $P$ of length $2\ell+1$ between $u_1$ and $v_1$ such that $V(P)\subseteq N_G(w)\setminus S$;
\item if $u_1,v_1\in (V_b\cap N_G(w))\setminus (L\cup S),$ then for any $\ell\in \{1,2,\ldots,k\},$ there exists an even path $P$ of length $2\ell$ between $u_1$ and $v_1$ such that $V(P)\subseteq N_G(w)\setminus S.$
\end{enumerate}
\end{lemma}

\begin{proof}
By Lemmas \ref{lem3.4} and \ref{lem3.9}, $d_G(w,V_a)\leq 3.$ Combining $w\notin L$ and Lemmas \ref{lem3.1}, we have
\begin{eqnarray}\label{eq20}
d_G(w,V_j)\geq d_G(w)-d_G(w,V_a)-(n-|V_a|-|V_j|)\geq \frac{n}{3}-7\sqrt{\varepsilon}n-3
\end{eqnarray}
for each $j\in \{b,c\}.$

(i) We consider the common neighbors of $w$ and $u_1$ in $V_c\setminus (S\cup L\cup\{v_1\}).$
Since $u_1\in (V_b\cap N_G(w))\setminus (L\cup S),$ by the technique leading to \eqref{eq20}, we have $d_G(u_1,V_c)\geq \frac{n}{3}-7\sqrt{\varepsilon}n-3.$ By Lemmas \ref{lem2.2} and \ref{lem3.1}, \eqref{eq20} and \eqref{eq2}, we have
\begin{eqnarray*}
& &\Big|\big(N_G(w,V_c)\cap N_G(u_1,V_c)\big)\setminus (S\cup L\cup\{v_1\})\Big|\\
&\geq& d_G(w,V_c)+d_G(u_1,V_c)-|V_c|-(2k+6)\\
&>& 2\left(\frac{n}{3}-7\sqrt{\varepsilon}n-3\right)-\left(\frac{1}{3}+2\sqrt{\varepsilon}\right)n-(2k+6)\\
&=& \left(\frac{1}{3}-16\sqrt{\varepsilon}\right)n-(2k+12)\\
&\geq& \ell
\end{eqnarray*}
for sufficiently large $n.$ Then we choose distinct vertices $v_2,v_3,\ldots, v_{\ell+1}\in (N_G(w,V_c)\cap N_G(u_1,V_c))\setminus (S\cup L\cup\{v_1\}).$ Next we consider the common neighbors of $w,v_1,v_2,\ldots, v_{\ell+1}$ in $V_b\setminus (S\cup\{u_1\}).$
Similarly to the proof of \eqref{eq20}, one can obtain that $d_G(v_i,V_b)\geq \frac{n}{3}-7\sqrt{\varepsilon}n-3$ for $i\in \{1,2,\ldots,\ell+1\}.$
Combining Lemmas \ref{lem2.2} and \ref{lem3.1}, \eqref{eq20} and \eqref{eq2}, we deduce that
\begin{eqnarray*}
& &\Big|\big(N_G(w,V_b)\cap N_G(v_1,V_b)\cap \cdots \cap N_G(v_{\ell+1},V_b)\big)\setminus (S\cup\{u_1\})\Big|\\
&>& (\ell+2)\left(\frac{n}{3}-7\sqrt{\varepsilon}n-3\right)-(\ell+1)\left(\frac{1}{3}+2\sqrt{\varepsilon}\right)n-(2k+3)\\
&=& \left(\frac{1}{3}-(9\ell+16)\sqrt{\varepsilon}\right)n-(2k+3)-3(\ell+2) \\
&\geq& \ell
\end{eqnarray*}
for large enough $n.$ Let $u_2,u_3,\ldots, u_{\ell+1}\in \big(N_G(w,V_b)\cap N_G(v_1,V_b)\cap \cdots \cap N_G(v_{\ell+1},V_b)\big)\setminus (S\cup\{u_1\}).$ Then we can find an odd path $u_1v_2u_2\cdots v_{\ell+1}u_{\ell+1}v_1$ in $N_G(w)\setminus S,$ as desired.

(ii) We consider the common neighbors of $w,u_1,v_1$ in $V_c\setminus (S\cup L).$ Similarly to \eqref{eq20}, $d_G(x,V_c)\geq \frac{n}{3}-7\sqrt{\varepsilon}n-3$ for any $x\in \{u_1,v_1\}.$ Combining Lemmas \ref{lem2.2} and \ref{lem3.1}, \eqref{eq20} and \eqref{eq2}, we deduce that
\begin{eqnarray*}
& &\Big|\big(N_G(w,V_c)\cap N_G(u_1,V_c)\cap N_G(v_1,V_c)\big)\setminus (S\cup L)\Big|\\
&\geq& \sum_{x\in\{w,u_1,v_1\}} d_G(x,V_c)-2|V_c|-(2k+5)\\
&>& 3\left(\frac{n}{3}-7\sqrt{\varepsilon}n-3\right)-2\left(\frac{1}{3}+2\sqrt{\varepsilon}\right)n-(2k+5)\\
&=& \left(\frac{1}{3}-25\sqrt{\varepsilon}\right)n-(2k+14)\geq \ell
\end{eqnarray*}
for sufficiently large $n.$ Let $v_2,v_3,\ldots, v_{\ell+1}\in \big(N_G(w,V_c)\cap N_G(u_1,V_c)\cap N_G(v_1,V_c)\big)\setminus (S\cup L).$ Next we consider the common neighbors of $w,v_2,\ldots, v_{\ell+1}$ in $V_b\setminus (S\cup\{u_1,v_1\}).$
Similarly to \eqref{eq20}, we have $d_G(v_i,V_b)\geq \frac{n}{3}-7\sqrt{\varepsilon}n-3$ for $i\in \{2,\ldots, \ell+1\}.$
By Lemmas \ref{lem2.2} and \ref{lem3.1}, \eqref{eq20} and \eqref{eq2}, we have
\begin{eqnarray*}
& &\Big|\big(N_G(w,V_b)\cap N_G(v_2,V_b)\cap \cdots \cap N_G(v_{\ell+1},V_b)\big)\setminus (S\cup\{u_1,v_1\})\Big| \\
&>& (\ell+1)\left(\frac{n}{3}-7\sqrt{\varepsilon}n-3\right)-\ell\left(\frac{1}{3}+2\sqrt{\varepsilon}\right)n-(2k+4)\geq \ell-1
\end{eqnarray*}
for large enough $n.$ Let $u_2,u_3,\ldots, u_{\ell}\in \big(N_G(w,V_b)\cap N_G(v_2,V_b)\cap \cdots \cap N_G(v_{\ell+1},V_b)\big)\setminus (S\cup\{u_1,v_1\}).$ Then we can find an even path $u_1v_2u_2\cdots u_{\ell}v_{\ell+1}v_1$ in $N_G(w)\setminus S,$ as desired.
\end{proof}

\begin{lemma}\label{lem3.12}
$\sum^3_{i=1}e(G[V_i])=1.$
\end{lemma}
\begin{proof}
\setcounter{case}{0}
Since $G$ is non-$3$-partite, $\sum^3_{i=1}e(G[V_i])\geq 1.$ Suppose to the contrary that $\sum^3_{i=1}e(G[V_i])\geq 2.$ According to Lemma \ref{lem3.10}, we have $\sum^2_{i=1}e(G[V_i])\geq 2.$ By Lemma \ref{lem3.9}, $1\leq |L|\leq 3.$ We have the following claims.

\begin{claim}\label{cla1}
$2\leq |L|\leq 3.$
\end{claim}

\begin{proof}
Suppose to the contrary that $L=\{u\}.$ Then $u\in V_1.$ Let $\{uv_1,uv_2,\ldots,uv_s\}\subseteq E(G[V_1]),$ where $s=d_G(u,V_1).$ Combining $\sum^3_{i=1}e(G[V_i])\geq 2$ and Lemma \ref{lem3.4}, we have $s=e(G[V_1])=\sum^3_{i=1}e(G[V_i])\geq 2.$
By the maximality of $e_G(V_1,V_2,V_3),$ we have
\begin{eqnarray}\label{eq21}
d_G(u,V_i)\geq d_G(u,V_1)\geq 2
\end{eqnarray}
for each $i\in\{2,3\}.$ We first prove that $E(N_G(u,V_2)\cap N_G(v_i,V_2), N_G(u,V_3))=\emptyset$ for each $i\in \{1,2,\ldots,s\}.$ Otherwise, by symmetry, suppose that there exist two vertices $w\in N_G(u,V_2)\cap N_G(v_1,V_2)$ and $w_1\in N_G(u,V_3)$ such that $ww_1\in E(G).$ Note that $L=\{u\}.$ Take $S=\{u\}.$ Then $v_1\in N_G(w,V_1)\setminus (L\cup S)$ and $w_1\in N_G(w,V_3)\setminus (L\cup S).$ By Lemma \ref{lem3.11}, there exists an odd path $P$ of length $2k-1$ between $w_1$ and $v_1$ such that $V(P)\subseteq N_G(w)\setminus S.$ Then we can find a wheel $W_{2k+2}$ in $G$ with center vertex $w$ and rim $v_1uw_1\cdots v_1,$ a contradiction. Hence $e(N_G(u,V_2)\cap N_G(v_i,V_2), N_G(u,V_3))=0$ for each $i\in \{1,2,\ldots, s\}.$ By symmetry, $e(N_G(u,V_3)\cap N_G(v_i,V_3), N_G(u,V_2))=0.$

Next we show that there exists some $j\in \{1,2,\ldots, s\}$ such that $N_G(u,V_2\cup V_3)\cap N_G(v_{j},V_2\cup V_3)\neq \emptyset.$ Suppose to the contrary that $N_G(u,V_2\cup V_3)\cap N_G(v_i,V_2\cup V_3) = \emptyset$ for each $i \in \{1,2,\ldots, s\}.$ Then $d_G(u,V_2\cup V_3)+d_G(v_i,V_2\cup V_3)\leq |V_2|+|V_3|.$ By \eqref{eq21}, $d_G(u,V_2\cup V_3) = d_G(u,V_2) + d_G(u,V_3) \geq 4.$ This implies $d_G(v_i,V_2\cup V_3)\leq |V_2|+|V_3|-4.$ It follows from $s\geq 2$ that
\begin{eqnarray*}
e(G)&\leq& e(K_{|V_1|,|V_2|,|V_3|})-2(|V_2|+|V_3|)+d_G(u,V_2\cup V_3)+d_G(v_1,V_2\cup V_3)\\
& & -(s-1)(|V_2|+|V_3|)+\sum^{s}_{i=2} d_G(v_i,V_2\cup V_3)+s\\
&\leq& e(K_{|V_1|-1,|V_2|,|V_3|})-3s+4\leq e(T_{n-1,3})-2,
\end{eqnarray*}
which contradicts \eqref{eq1}. Hence there exists some $j\in \{1,2,\ldots, s\}$ such that $N_G(u,V_2\cup V_3)\cap N_G(v_{j},V_2\cup V_3)\neq \emptyset.$

Without loss of generality, suppose that $w'\in N_G(u,V_2)\cap N_G(v_1,V_2).$ Recall that $E(N_G(u,V_2)\cap N_G(v_1,V_2), N_G(u,V_3))=\emptyset.$ Let $G'=G-\{uw'\}+\{w'x:x\in N_G(u,V_3)\}.$ It follows from \eqref{eq21} that $d_G(u,V_3)\geq 2.$ Then $e(G')>e(G).$
Next we prove that $G'$ is $W_{2k+2}$-free. Otherwise, $w'$ must be a vertex of a $W_{2k+2},$ say $W,$ in $G'.$ By $L=\{u\}$ and Lemma \ref{lem3.4}, $N_{G'}(w')\subseteq (V_1\cup V_3)\setminus\{u\},$ which implies that $G'[N_{G'}(w')]$ is bipartite. Then $w'$ cannot be the center of $W_{2k+2}$. Therefore, $w'$ must lie on the cycle of $W.$ Suppose that $w'_1, w'_2, w'_3$ are neighbors of $w'$ in $W.$ Then $w'_1, w'_2, w'_3\notin V_2\cup L.$ For $w'_i\in V_j,$ we have $d_G(w'_i,V_j)\leq 1.$ It follows that
\begin{eqnarray*}
d_G(w'_i,V_2)&=&d_G(w'_i)-d_G(w'_i,V_1)-d_G(w'_i,V_3)\\
&\geq& d_G(w'_i)-(n-|V_2|-|V_j|)-d_G(w'_i,V_j)\\
&\geq& \left(\frac{1}{3}-7\sqrt{\varepsilon}\right)n-1.
\end{eqnarray*}
By Lemmas \ref{lem2.2} and \ref{lem3.1} and \eqref{eq2}, we have
\begin{eqnarray*}
\left|\left(\bigcap^3_{i=1}N_G(w'_i,V_2)\right)\setminus V(W)\right|&\geq& 3\left(\frac{1}{3}-7\sqrt{\varepsilon}\right)n-2\left(\frac{1}{3}+2\sqrt{\varepsilon}\right)n-(2k+5)\\
&=& \left(\frac{1}{3}-25\sqrt{\varepsilon}\right)n-2k-5>0
\end{eqnarray*}
for sufficiently large $n.$ Hence there exists a vertex $w''\in \bigcap^3_{i=1} N_G(w'_i,V_2)\setminus V(W).$ Therefore, we can find a wheel $W_{2k+2}$ of $G$ by replacing $w'$ with $w''$ in $W,$ a contradiction. Hence $G'$ is $W_{2k+2}$-free.

Finally, we show that $G'$ is non-$3$-partite. Suppose to the contrary that $G'$ is $3$-partite. Let $\{V'_1, V'_2, V'_3\}$ be a partition of $V(G')$
such that $\sum_{i=1}^3e(G'[V'_i])=0.$
Let $R=L\cup\{w'\}=\{u,w'\}.$ According to the construction of $G',$
$E(G[V(G)\setminus R])=E(G'[V(G')\setminus R]).$ Therefore, by the argument in the proof of Lemma \ref{lem3.5} (ii), with $L$ replaced by $R,$ we assume that $V_1\setminus L\subseteq V'_1,$ $V_2\setminus (L\cup\{w'\})\subseteq V'_2$ and $V_3\setminus L\subseteq V'_3.$
Since $L=\{u\},$ we have $V_1\setminus\{u\}\subseteq V'_1,$ $V_2\setminus\{w'\}\subseteq V'_2$ and $V_3\subseteq V'_3.$ Note that $uv_i\in E(G')$ and $v_i\in V'_1$ for $i\in \{1,2,\ldots, s\}.$ Then $d_{G'}(u,V'_1)\geq s\geq 2.$ By \eqref{eq21}, we have $d_{G'}(u,V'_2) \geq d_G(u,V_2)-1 \geq 1$ and $d_{G'}(u,V'_3) = d_G(u,V_3) \geq 2.$ Hence $G'$ is a non-$3$-partite $W_{2k+2}$-free graph.

In summary, we have proved that $G'$ is a non-$3$-partite $W_{2k+2}$-free graph with $e(G') > e(G).$ This contradicts the maximality of $e(G).$ Therefore, $|L|\geq 2.$
\end{proof}

\begin{claim}\label{cla2}
$e(G[V_1])=1$ and $e(G[V_2])\geq 1.$
\end{claim}
\begin{proof}
Recall that $\sum^2_{i=1}e(G[V_i])\geq 2.$ If $e(G[V_1])=0,$ then $e(G[V_2])\geq 2.$ Let $uv\in E(G[V_2])$ and $w'\in V_1 \cap L.$ By Lemma \ref{lem3.4}, $u\in L$ or $v\in L.$ Without loss of generality, we may assume that $u\in L.$ Recall that $|V_2\cap L|\leq 1.$ Then $d_G(u,V_2)=e(G[V_2])\geq 2.$ Let $G'=G-\{w'x:x\in N_G(w')\}+\{w'x:x\in V(G)\setminus(V_1\cup L)\}.$
Combining Lemmas \ref{lem3.5} and \ref{lem3.7}, we have $N_G(u,V_1)=\{w'\}.$ Then $d_G(u,V_2)>d_G(u,V_1),$ which contradicts the maximality of $e_G(V_1,V_2,V_3).$
Suppose that $e(G[V_1])\geq 2.$ If $|V_1\cap L|=1,$ then by Claim \ref{cla1}, $|V_2\cap L|=1.$ Let $V_1\cap L=\{u\}$ and $V_2\cap L=\{w'\}.$ By the technique above, we have $N_G(u,V_2)=\{w'\}.$ Then $d_G(u,V_1)>d_G(u,V_2),$ which contradicts the maximality of $e_G(V_1,V_2,V_3).$
If $|V_1\cap L|\geq 2,$ then combining Lemma \ref{lem3.4} and $e(G[V_1])\geq 2,$ we can find a vertex $w'\in V_1\cap L$ and an edge $uv\in E(G[V_1])$ such that $w'\notin \{u,v\},$ which contradicts Lemma \ref{lem3.8}.
\end{proof}

Recall that $|V_2\cap L|\leq 1.$ It follows from $e(G[V_2])\geq 1$ and Lemma \ref{lem3.4} that $|V_2\cap L|= 1.$
By Claim \ref{cla2}, $|V_2\cap L|=1,$ and by Lemma \ref{lem3.8}, $|V_1\cap L|=1.$
Combining Claim \ref{cla2} and Lemma \ref{lem3.4}, there exist two edges $u'_1v'_1\in E(G[V_1])$ and $u'_2v'_2\in E(G[V_2]),$ where $\{u'_1,u'_2\}\subseteq L$ and $v'_2\notin L.$ By Lemma \ref{lem3.8}, $|V_1\cap L|=1.$ Hence $L=\{u'_1,u'_2\}.$ Let
\begin{eqnarray*}
G'=G-\{u'_2x:x\in N_G(u'_2)\}+\{u'_2x:x\in V(G)\setminus(V_1\cup L)\}.
\end{eqnarray*}
By Lemma \ref{lem3.5}, $G'$ is a $3$-partite graph and there exists a partition $\{V'_1, V'_2, V'_3\}$ of $V(G')$ such that $\sum^3_{i=1}e(G'[V'_i])=0$ and $V_i\setminus L\subseteq V'_i$ for each $i\in \{1,2,3\}.$ Since $u'_1v'_1\in E(G[V_1])$ and $V_1\cap L=\{u'_1\},$ we have $u'_1\notin V'_1.$ It follows from Lemma \ref{lem3.7} that $N_G(u'_1,V_2)=\{u'_2\}$ and $N_G(u'_1,V_1)=\{v'_1\}.$ Similarly, $N_G(u'_2,V_1)=\{u'_1\}$ and $N_G(u'_2,V_2)=\{v'_2\}.$ So $e_G\big((V_1\setminus\{u'_1\})\cup\{u'_2\}, (V_2\setminus\{u'_2\})\cup\{u'_1\}, V_3\big)>e_G(V_1,V_2,V_3),$ a contradiction.
\end{proof}

\begin{figure}[H]
\centering
\includegraphics[width=0.5\textwidth]{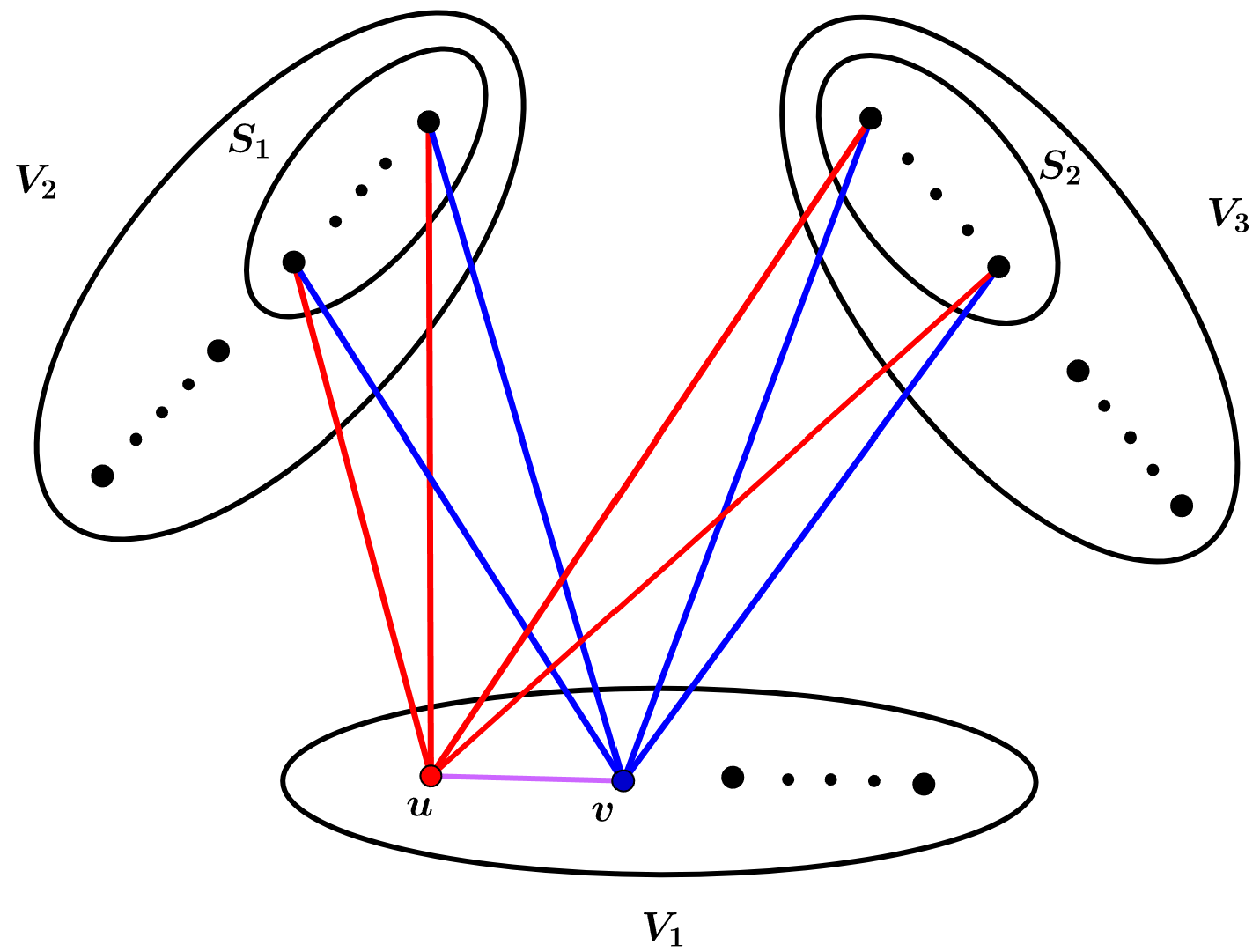}
\caption{Graph $G.$}\label{fig2}
\end{figure}

According to Lemma \ref{lem3.12}, $\sum_{i=1}^3 e(G[V_i]) =1.$ Let $\bigcup_{i=1}^3 E(G[V_i])=\{uv\}.$ By Lemma \ref{lem3.10}, $V_3\cap L=\emptyset.$ It follows from Lemma \ref{lem3.4} that $E(G[V_3]) = \emptyset.$ By relabeling the partition, we obtain $uv\in E(G[V_1])$ and $|N_G(u,V_2)\cap N_G(v,V_2)| \geq |N_G(u,V_3)\cap N_G(v,V_3)|.$ Then $V_1\cap L\neq \emptyset.$ According to Lemma \ref{lem3.10}, $V_2\cap L=\emptyset$ or $V_3\cap L=\emptyset.$
Let
\begin{eqnarray*}
S_1 =N_G(u,V_2)\cap N_G(v,V_2) \quad\text{and}\quad S_2 =N_G(u,V_3)\cap N_G(v,V_3).
\end{eqnarray*}

\begin{lemma}\label{lem3.13}
$|S_1|\geq 1$ and $|S_2|\leq 1.$
\end{lemma}
\begin{proof}
First, we show that $|S_1|\geq 1.$ Suppose to the contrary that $|S_1| = 0.$ Since $|S_1| \geq |S_2|,$ we have $|S_2| = 0.$ Consequently, $u$ and $v$ have no common neighbors in $V_2 \cup V_3,$ which implies that $d_G(u,V_2\cup V_3) + d_G(v,V_2\cup V_3) \leq |V_2|+|V_3|.$ By Lemma \ref{lem3.12}, we have
\begin{eqnarray*}
e(G) &\leq& e(K_{|V_1|, |V_2|, |V_3|}) -2(|V_2|+|V_3|)+ d_G(u,V_2\cup V_3) + d_G(v,V_2\cup V_3)+ 1 \\
&\leq& e(K_{|V_1|,|V_2|,|V_3|})-(|V_2|+|V_3|)+1 \\
&\leq& e(T_{n-1,3})+ 1,
\end{eqnarray*}
which contradicts \eqref{eq1}. Therefore, $|S_1|\geq 1.$

Next, we prove that  $|S_2|\leq 1.$ Suppose to the contrary that $|S_2|\geq 2.$ Since $|S_1|\geq |S_2|,$ we have $|S_1|\geq 2.$
\begin{claim}\label{cla3}
$e(S_1, S_2)\geq 1.$
\end{claim}
\begin{proof}
Suppose to the contrary that $e(S_1, S_2) = 0.$ Note that
\begin{eqnarray}\label{eq22}
d_G(u,V_2\cup V_3) + d_G(v,V_2\cup V_3) \leq |V_2| + |V_3| + |S_1| + |S_2|.
\end{eqnarray}
Since $|S_1|\geq 2$ and $|S_2|\geq 2,$ we have $(|S_1|-1)(|S_2|-1) \geq 1,$ which implies
\begin{eqnarray}\label{eq23}
|S_1|+|S_2|-|S_1||S_2|\leq 0.
\end{eqnarray}
Combining \eqref{eq22} and \eqref{eq23}, we have
\begin{eqnarray*}
e(G) &\leq& e(K_{|V_1|,|V_2|,|V_3|}) - |S_1||S_2| - 2(|V_2|+|V_3|) + (|V_2|+|V_3|+|S_1|+|S_2|) + 1 \\
&=& e(K_{|V_1|,|V_2|,|V_3|}) - |S_1||S_2| - (|V_2|+|V_3|) + |S_1| + |S_2| + 1 \\
&=& e(K_{|V_1|-1,|V_2|,|V_3|}) - |S_1||S_2| + |S_1| + |S_2| + 1\\
&\leq& e(T_{n-1,3}) + 1,
\end{eqnarray*}
which contradicts \eqref{eq1}.
\end{proof}

\begin{claim}\label{cla4}
$L\subseteq \{u,v\}.$
\end{claim}
\begin{proof}
We first prove that $L\subset V_1.$ Suppose to the contrary that there exists a vertex $w\in (V_2\cup V_3)\cap L.$ Let $G' = G-\{wx: x \in N_G(w)\} + \{wx: x \in V(G) \setminus (V_1 \cup L)\}.$ By Lemma \ref{lem3.5}, $G'$ is a $3$-partite graph and there exists a partition $\{V'_1, V'_2, V'_3\}$ of $V(G')$ such that $\sum^3_{i=1}e(G'[V'_i])=0$ and $V_i\setminus L\subseteq V'_i$ for each $i\in \{1,2,3\}.$ Since $uv\in E(G[V_1]),$ without loss of generality, suppose that $u\in V'_j$ with $j\in \{2,3\}.$ According to Lemma \ref{lem3.7}, we have $N_G(u,V_j)=\{w\},$ which contradicts that $|S_1|\geq |S_2|\geq 2.$ Hence $L\subset V_1.$ By Lemma \ref{lem3.8}, $(V_i\cap L)\setminus \{u,v\}=\emptyset.$ Therefore, $L\subseteq \{u,v\}.$
\end{proof}
By Claim \ref{cla3}, we assume that $w_1 \in S_1$ and $w_2 \in S_2$ such that $w_1w_2\in E(G).$

\begin{claim}\label{cla5}
$N_G(w_2,V_2\setminus \{w_1\})\cap N_G(u,V_2\setminus \{w_1\})=\emptyset$ and $N_G(w_1,V_3\setminus \{w_2\})\cap N_G(u,V_3\setminus \{w_2\})=\emptyset.$
\end{claim}
\begin{proof}
Suppose to the contrary that there exists a vertex $w_3\in N_G(w_2,V_2\setminus \{w_1\})\cap N_G(u,V_2\setminus \{w_1\}).$ It follows from Claim \ref{cla4} that $w_2\notin L.$ Let $S=\{u,v\}.$ Then $w_1,w_3\in N_G(w_2, V_2)\setminus (L\cup S).$ By Lemma \ref{lem3.11}, there exists an even path $P$ of length $2k-2$ between $w_1$ and $w_3$ such that $V(P)\subseteq N_G(w_2)\setminus S.$ Hence we find a copy of $W_{2k+2}$ with center vertex $w_2$ and rim $w_3uvw_1\cdots w_3$ in $G,$ a contradiction. By the technique above, $N_G(w_1,V_3\setminus \{w_2\})\cap N_G(u,V_3\setminus \{w_2\})=\emptyset.$
\end{proof}

Since $w_2\notin L,$ we have $d_G(w_2,V_2\setminus \{w_1\})\geq \big(\frac{2}{3}-5\sqrt{\varepsilon}\big)n-|V_1|-1.$ Combining Claim \ref{cla5} and Lemma \ref{lem3.1}, we have
\begin{eqnarray*}
d_G(u,V_2\setminus \{w_1\})\leq (|V_2|-1)-d_G(w_2,V_2\setminus \{w_1\})\leq (|V_1|+|V_2|)-\left(\frac{2}{3}-5\sqrt{\varepsilon}\right)n\leq 7\sqrt{\varepsilon}n,
\end{eqnarray*}
which implies that $d_G(u,V_2)\leq 7\sqrt{\varepsilon}n+1.$
Similarly, $d_G(u,V_3)\leq 7\sqrt{\varepsilon}n+1.$ Then $d_G(u,V_2\cup V_3)\leq 14\sqrt{\varepsilon}n+2.$
By symmetry, $d_G(v,V_2\cup V_3)\leq 14\sqrt{\varepsilon}n+2.$ Hence
\begin{eqnarray*}
e(G) &\leq& e(K_{|V_1|,|V_2|,|V_3|})-2(|V_2|+|V_3|)+d_G(u,V_2\cup V_3)+d_G(v,V_2\cup V_3)+1\\
&\leq& e(T_{n-1,3})-(|V_2|+|V_3|)+28\sqrt{\varepsilon}n+5 < e(T_{n-1,3})+2
\end{eqnarray*}
for sufficiently large $n,$ which contradicts \eqref{eq1}.
Hence $|S_2|\leq 1.$
\end{proof}

\begin{figure}[H]
\centering
\includegraphics[width=0.5\textwidth]{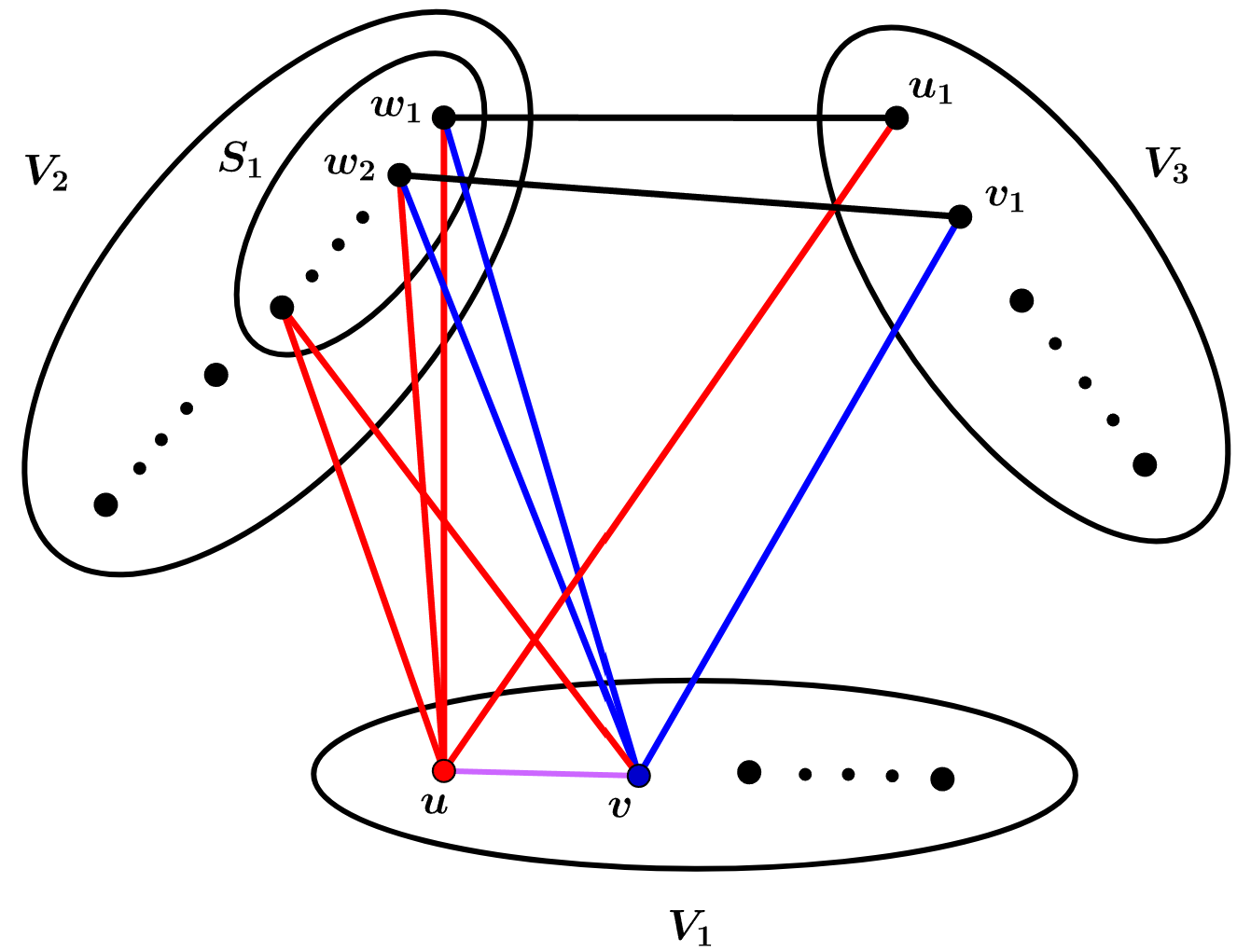}
\caption{Graph $G.$}\label{fig3}
\end{figure}

\begin{lemma}\label{lem3.14}
$E_G(S_1,N_G(u,V_3))=\emptyset$ or $E_G(S_1,N_G(v,V_3))=\emptyset.$
\end{lemma}

\begin{proof}
\setcounter{case}{0}
Suppose to the contrary that $E_G(S_1,N_G(u,V_3))\neq \emptyset$ and $E_G(S_1,N_G(v,V_3))\neq \emptyset.$
Let $w_1u_1\in E_G(S_1,N_G(u,V_3))$ and $w_2v_1\in E_G(S_1,N_G(v,V_3)),$ where $w_1,w_2\in S_1,$ $u_1\in N_G(u,V_3)$ and $v_1\in N_G(v,V_3)$ (see Fig. \ref{fig3}). We divide our proof into the following cases.

\begin{case}
$w_1\neq w_2$ and $u_1\neq v_1.$
\end{case}
In this case, $|S_1|\geq 2.$ If $V_2\cap L\neq \emptyset,$ then we choose $w\in V_2\cap L.$ Let $G' = G - \{wx: x \in N_G(w)\} + \{wx: x \in V(G) \setminus (V_2 \cup L)\}.$ By Lemmas \ref{lem3.5} and \ref{lem3.7}, $N_G(u,V_2)=\{w\}$ or $N_G(v,V_2)=\{w\},$ which contradicts $|S_1|\geq 2.$
Then $V_2\cap L=\emptyset,$ which implies that $w_1,w_2\notin L.$ It follows from Lemma \ref{lem3.8} that $V_1\cap L\subseteq \{u,v\}.$ Note that $u\in L$ or $v\in L.$ By symmetry, we assume that $u\in L.$ We consider the following two subcases.

\begin{subcase}
$v\notin L.$
\end{subcase}
Suppose that $u_1\notin L.$
Recall that $w_1\notin L.$ Let $S=\{u\}.$ Then $v\in N_G(w_1,V_1)\setminus (L\cup S)$ and $u_1\in N_G(w_1,V_3)\setminus (L\cup S).$ By Lemma \ref{lem3.11}, there exists an odd path of length $2k-1$ between $u_1$ and $v$ in $N_G(w_1)\setminus S.$ Then we can find a $W_{2k+2}$ in $G$ with center vertex $w_1$ and rim $u_1uv\cdots u_1,$ a contradiction. Hence $u_1\in L.$ By Lemma \ref{lem3.8}, $V_3\cap L=\{u_1\}.$ Then $L=\{u,u_1\}.$
\begin{claim}\label{cla6}
$N_G(u,V_3)=\{u_1\}$ and $S_2=\emptyset.$
\end{claim}
\begin{proof}
Let $G' = G - \{u_1x: x \in N_G(u_1)\} + \{u_1x: x \in V(G) \setminus (V_3 \cup L)\}.$ Combining $u\in L,$ $v\notin L$ and Lemmas \ref{lem3.5} and \ref{lem3.7}, we have $N_G(u,V_3)=\{u_1\}.$
Next we prove that $vu_1\notin E(G).$
Suppose to the contrary that $vu_1\in E(G).$ Let $S=\{u,u_1\}.$ Since $v,w_1,w_2\notin L,$ by Lemma \ref{lem3.11}, there exists an even path $P$ of length $2k-2$ between $w_1$ and $w_2$  such that $V(P)\subseteq N_G(v)\setminus S.$ Hence we can find a $W_{2k+2}$ in $G$ with center vertex $v$ and rim $w_1u_1uw_2\cdots w_1,$ a contradiction.
Hence $S_2=\emptyset.$
\end{proof}

\begin{claim}\label{cla7}
$N_G(w_1,V_1)\cap N_G(u_1,V_1)=\{u\}.$
\end{claim}
\begin{proof}
By Claim \ref{cla6}, $u_1v\notin E(G).$ We claim that $(N_G(w_1,V_1)\cap N_G(u_1,V_1))\setminus \{u,v\}=\emptyset.$ Suppose to the contrary that there exists a vertex $w\in (N_G(w_1,V_1)\cap N_G(u_1,V_1))\setminus \{u,v\}.$ By the technique above, we can find a $W_{2k+2}$ in $G$ with center vertex $w_1$ and rim $wu_1uv\cdots w,$ a contradiction. Hence $N_G(w_1,V_1)\cap N_G(u_1,V_1)=\{u\}.$
\end{proof}

We claim that $N_G(u_1,V_1\setminus\{u,v\})\neq\emptyset.$
Suppose otherwise. Since $vu_1\notin E(G)$ by Claim \ref{cla6} and $uu_1\in E(G)$, we have $N_G(u_1,V_1)=\{u\}.$
It follows from Claim \ref{cla6} that $N_G(u,V_3)=\{u_1\}.$
Then $e_G((V_1\setminus\{u\})\cup\{u_1\},V_2,(V_3\setminus\{u_1\})\cup\{u\})>e_G(V_1,V_2,V_3),$ a contradiction.
Let $w''\in N_G(u_1,V_1\setminus\{u,v\}).$ Recall that $L=\{u,u_1\}.$ Then $w''\notin L.$

\begin{claim}\label{cla8}
For every $w'\in V_2\setminus\{w_1,w_2\},$ we have $\{w'u,w'v,w'u_1,w'w''\}\setminus E(G)\neq \emptyset.$
\end{claim}
\begin{proof}
Suppose otherwise. Let $S=\{u,u_1\}.$ Note that $w',v,w''\notin L.$ By Lemma \ref{lem3.11}, there exists an even path $P$ of length $2k-2$ between $v$ and $w''$ such that $V(P)\subseteq N_G(w')\setminus S.$ Hence we can find a $W_{2k+2}$ in $G$ with center vertex $w'$ and rim $w''u_1uv\cdots w'',$ a contradiction.
\end{proof}

According to Claim \ref{cla8}, we have $e_{\overline{G}}(V_2\setminus\{w_1,w_2\}, \{u,v,u_1,w''\})\geq |V_2|-2.$ Combining Claims \ref{cla6} and \ref{cla7}, we have
\begin{eqnarray*}
e(G)&\leq& e(K_{|V_1|,|V_2|,|V_3|})-\big(e_{K_{|V_1|,|V_2|,|V_3|}}(\{w_1,u_1\},V_1)+e_{K_{|V_1|,|V_2|,|V_3|}}(\{u,v\},V_3)-2\big)+(d_G(w_1,V_1)\\
& & +d_G(u_1,V_1)+d_G(u,V_3)+ d_G(v,V_3)-1)-e_{\overline{G}}(V_2\setminus\{w_1,w_2\}, \{u,v,u_1,w''\})+1\\
&\leq& e(K_{|V_1|,|V_2|,|V_3|})-|V_1|-|V_2|-|V_3|+5< e(T_{n,3})-\left(n-\left\lceil\frac{n}{3}\right\rceil\right)+2
\end{eqnarray*}
for sufficiently large $n,$ which contradicts \eqref{eq1}.

\begin{subcase}
$v\in L.$
\end{subcase}
Recall that $\{u,v\}\subseteq L.$ By Lemma \ref{lem3.9}, $|V_3\cap L|\leq 1.$ Then at most one of $u_1$ and $v_1$ belongs to $L.$
\begin{claim}\label{cla9}
Either $u_1\in L$ or $v_1\in L.$
\end{claim}
\begin{proof}
Suppose to the contrary that $u_1,v_1\notin L.$ Then we prove that
\begin{eqnarray}\label{eq24}
(N_G(u,V_3)\cap N_G(w_2,V_3))\setminus (L\cup \{v_1\})=\emptyset.
\end{eqnarray}
Suppose otherwise. Then there exists a vertex $w_0\in (N_G(u,V_3)\cap N_G(w_2,V_3))\setminus (L\cup \{v_1\}).$ Combining $w_0,w_2,v_1\notin L$ and Lemma \ref{lem3.11}, there exists an even path $P$ of length $2k-2$ between $w_0$ and $v_1$  such that $V(P)\subseteq N_G(w_2)\setminus \{u,v\}.$ Hence we can find a $W_{2k+2}$ in $G$ with center vertex $w_2$ and rim $v_1vuw_0\cdots v_1,$ which contradicts the assumption that $G$ is $W_{2k+2}$-free.
Combining \eqref{eq24} and Lemma \ref{lem2.2}, we have
\begin{eqnarray}\label{eq25}
e_G(\{u,w_2\}, V_3\setminus (L\cup \{v_1\}))\leq |V_3\setminus (L\cup \{v_1\})|.
\end{eqnarray}
By the technique above, we have $(N_G(v,V_3)\cap N_G(w_1,V_3))\setminus (L\cup \{u_1\})=\emptyset,$ and hence
\begin{eqnarray}\label{eq26}
e_G(\{v,w_1\}, V_3\setminus (L\cup \{u_1\}))\leq |V_3\setminus (L\cup \{u_1\})|.
\end{eqnarray}
Moreover, we show that $|S_1|\geq \frac{1}{2}|V_2|.$ Suppose to the contrary that $|S_1|< \frac{1}{2}|V_2|.$
It follows from Lemma \ref{lem2.2} that
\begin{eqnarray}\label{eq27}
e_G(\{u,v\}, V_2)<\frac{3}{2}|V_2|.
\end{eqnarray}
Combining \eqref{eq25}, \eqref{eq26} and \eqref{eq27} and Lemma \ref{lem3.1}, we have
\begin{eqnarray*}
e(G)&\leq& e(K_{|V_1|,|V_2|,|V_3|})-2|V_3\setminus (L\cup \{v_1\})|+e_G(\{u,w_2\}, V_3\setminus (L\cup \{v_1\}))-2|V_3\setminus (L\cup \{u_1\})|\\
& & +e_G(\{v,w_1\}, V_3\setminus (L\cup \{u_1\}))-2|V_2|+e_G(\{u,v\}, V_2)+1\\
&<& e(K_{|V_1|,|V_2|,|V_3|})-2(|V_3|-2)-\frac{1}{2}|V_2|+1\\
&<& e(T_{n,3})-\left(n-\left\lceil\frac{n}{3}\right\rceil\right)+2
\end{eqnarray*}
for sufficiently large $n,$ which contradicts \eqref{eq1}. Hence $|S_1|\geq \frac{1}{2}|V_2|.$ Recall that $V_2\cap L=\emptyset$ and $u_1,v_1\notin L.$
Combining Lemmas \ref{lem2.2} and \ref{lem3.1} and \eqref{eq2}, we have
\begin{eqnarray*}
|S_1\cap N_G(u_1,V_2)\cap N_G(v_1,V_2)|&\geq& \frac{1}{2}|V_2|+d_G(u_1,V_2)+d_G(v_1,V_2)-2|V_2|\\
&\geq& d_G(u_1)+d_G(v_1)-2|V_1|-\frac{3}{2}|V_2|\\
&>& \left(\frac{1}{6}-17\sqrt{\varepsilon}\right)n>0.
\end{eqnarray*}
Suppose that $w'\in S_1\cap N_G(u_1,V_2)\cap N_G(v_1,V_2).$
Combining $w',u_1,v_1\notin L$ and Lemma \ref{lem3.11}, there exists an even path $P$ of length $2k-2$ between $u_1$ and $v_1$ such that $V(P)\subseteq N_G(w')\setminus \{u,v\}.$ Then we can find a $W_{2k+2}$ in $G$ with center vertex $w'$ and rim $u_1uvv_1\cdots u_1,$ a contradiction.
\end{proof}

According to Claim \ref{cla9}, without loss of generality, we suppose that $u_1\in L,$ which implies that $L=\{u,v,u_1\}.$
Let $G' = G - \{u_1x: x \in N_G(u_1)\} + \{u_1x: x \in V(G) \setminus (V_3 \cup L)\}.$ By Lemma \ref{lem3.5}, $G'$ is a $3$-partite graph and there exists a partition $\{V'_1, V'_2, V'_3\}$ of $V(G')$ such that $\sum^3_{i=1}e(G'[V'_i])=0$ and $V_i\setminus L\subseteq V'_i$ for each $i\in\{1,2,3\}.$ Since $uv,vv_1\in E(G),$ we have $u\in V'_{3}$ in $G'.$ It follows from Lemma \ref{lem3.7} that $N_G(u,V_3)=\{u_1\},$ which implies that
\begin{eqnarray}\label{eq28}
d_G(u,V_3)=1.
\end{eqnarray}

\begin{claim}\label{cla10}
$(N_G(v,V_3)\cap N_G(w_1,V_3))\setminus \{u_1,v_1\}=\emptyset.$
\end{claim}
\begin{proof}
Suppose to the contrary that there exists a vertex $w_3\in (N_G(v,V_3)\cap N_G(w_1,V_3))\setminus \{u_1,v_1\}.$ Then $w_3\notin L.$ Suppose that $N_G(w_1,V_1\setminus \{u,v\})\cap N_G(u_1,V_1\setminus \{u,v\})\cap N_G(w_3,V_1\setminus \{u,v\})=\emptyset.$ By Lemma \ref{lem2.2}, we have
\begin{eqnarray}\label{eq29}
d_G(w_1,V_1\setminus \{u,v\})+d_G(u_1,V_1\setminus \{u,v\})+d_G(w_3,V_1\setminus \{u,v\})\leq 2|V_1\setminus \{u,v\}|.
\end{eqnarray}
Combining $v\in L,$ \eqref{eq28} and \eqref{eq29}, we have
\begin{eqnarray*}
e(G)&\leq& e(K_{|V_1|,|V_2|,|V_3|})-|V_3|+d_G(u,V_3)-3|V_1\setminus\{u,v\}|+d_G(w_1,V_1\setminus \{u,v\})\\
& & +d_G(u_1,V_1\setminus \{u,v\})+d_G(w_3,V_1\setminus \{u,v\})-(|V_2|+|V_3|)+d_G(v)+1\\
&\leq& e(K_{|V_1|,|V_2|,|V_3|})-(|V_3|-1)-(|V_1|-2)-3\sqrt{\varepsilon}n+1\\
&\leq& e(T_{n-1,3})-3\sqrt{\varepsilon}n+4<e(T_{n-1,3})+2
\end{eqnarray*}
for sufficiently large $n,$ contradicting \eqref{eq1}. Hence there exists a vertex $w_4\in N_G(w_1,V_1\setminus \{u,v\})\cap N_G(u_1,V_1\setminus \{u,v\})\cap N_G(w_3,V_1\setminus \{u,v\}).$ If $k=2,$ then we can find a $W_{2k+2}$ in $G$ with center vertex $w_1$ and rim $w_3vuu_1w_4w_3,$ which contradicts that $G$ is $W_{2k+2}$-free.
Suppose that $k\geq 3.$ Note that $w_1, w_3,w_4\notin L.$ By Lemma \ref{lem3.11}, there exists an odd path $P$ of length $2k-3$ between $w_3$ and $w_4$  such that $V(P)\subseteq N_G(w_1)\setminus \{u,v,u_1\}.$ Then we can find a $W_{2k+2}$ in $G$ with center vertex $w_1$ and rim $w_3vuu_1w_4\cdots w_3,$ a contradiction.
\end{proof}

By Claim \ref{cla10}, we obtain that
\begin{eqnarray}\label{eq30}
d_G(v,V_3\setminus \{u_1,v_1\})+d_G(w_1,V_3\setminus \{u_1,v_1\})\leq |V_3|-2.
\end{eqnarray}
Combining $u_1\in L,$ \eqref{eq28} and \eqref{eq30}, we have
\begin{eqnarray*}
e(G)&\leq& e(K_{|V_1|,|V_2|,|V_3|})-|V_3|-2|V_3\setminus \{u_1,v_1\}|+d_G(u,V_3)+d_G(v,V_3\setminus \{u_1,v_1\})\\
& & +d_G(w_1,V_3\setminus \{u_1,v_1\})-(|V_1|+|V_2|)+d_G(u_1)+1\\
&\leq& e(K_{|V_1|,|V_2|,|V_3|})-2|V_3|+3-(|V_1|+|V_2|)+d_G(u_1)+1\\
&=& e(K_{|V_1|-1,|V_2|,|V_3|})-(|V_1|+|V_3|)+d_G(u_1)+4\\
&\leq& e(T_{n-1,3})-3\sqrt{\varepsilon}n+4<e(T_{n-1,3})+2
\end{eqnarray*}
for large enough $n,$ which contradicts \eqref{eq1}.

\begin{case}\label{ca2}
$w_1=w_2$ and $u_1\neq v_1.$
\end{case}

Suppose that $(V_2\cup V_3)\cap L=\emptyset.$ Combining $w_1,u_1,v_1\notin L$ and Lemma \ref{lem3.11}, there exists an even path of length $2k-2$ between $u_1$ and $v_1$ such that $V(P)\subseteq N_G(w_1)\setminus \{u,v\}.$ Hence one can find a $W_{2k+2}$ in $G$ with center vertex $w_1$ and rim $u_1uvv_1\cdots u_1,$ a contradiction. Hence $(V_2\cup V_3)\cap L\neq \emptyset.$ Then there exists some $j\in \{2,3\}$ such that $V_j\cap L\neq \emptyset$ and we assume that $w\in V_j\cap L.$

\begin{claim}\label{cla11}
$|S_1|=1.$
\end{claim}

\begin{proof}
Suppose to the contrary that $|S_1|\geq 2.$ Similarly to Case $1,$ we have $V_2\cap L=\emptyset.$ Since $(V_2\cup V_3)\cap L\neq\emptyset,$ we may choose $w\in V_3\cap L.$ If $e_G(S_1\setminus\{w_1\},\{u_1,v_1\})\neq 0,$ then this reduces to Case $1.$ Hence $e_G(S_1\setminus\{w_1\},\{u_1,v_1\})=0.$
Since $u_1\neq v_1,$ we choose $z\in\{u_1,v_1\}\setminus\{w\}.$ Then $e_G(S_1\setminus\{w_1\},\{z\})=0,$ which implies that $$e_{\overline{G}}(S_1\setminus\{w_1\},\{z\})=|S_1|-1.$$ Note that $e_G(\{u,v\},V_2\cup V_3) \leq |V_2|+|V_3|+|S_1|+|S_2|.$
Combining $w\in V_3\cap L$ and Lemmas \ref{lem3.1},
\ref{lem3.12} and \ref{lem3.13}, we have
\begin{eqnarray*}
e(G)&\leq& e(K_{|V_1|,|V_2|,|V_3|})-2(|V_2|+|V_3|)+e_G(\{u,v\},V_2\cup V_3)-(n-|V_3|-2)\\
& &+d_G(w,V(G)\setminus \{u,v\})-e_{\overline{G}}(S_1\setminus\{w_1\},\{z\})+1\\
&\leq& e(K_{|V_1|,|V_2|,|V_3|})-(|V_2|+|V_3|)+|S_2|-(n-|V_3|)+d_G(w)+4\\
&\leq& e(K_{|V_1|-1,|V_2|,|V_3|})-3\sqrt{\varepsilon}n+5\\
&\leq& e(T_{n-1,3})-3\sqrt{\varepsilon}n+5\\
&<& e(T_{n-1,3})+2
\end{eqnarray*}
for sufficiently large $n,$ contradicting \eqref{eq1}.
\end{proof}

Combining Claim \ref{cla11} and $|S_2|\leq 1,$ we have $$d_G(u,V_2\cup V_3)+d_G(v,V_2\cup V_3)\leq |V_2|+|V_3|+2.$$ It follows from $w\in V_j\cap L$ that
\begin{eqnarray*}
e(G)&\leq& e(K_{|V_1|,|V_2|,|V_3|})-2(|V_2|+|V_3|)+e_G(\{u,v\},V_2\cup V_3)-(n-|V_j|-2)\\
&&+d_G(w,V(G)\setminus \{u,v\})+1\\
&\leq& e(K_{|V_1|-1,|V_2|,|V_3|})-(n-|V_j|)+d_G(w,V(G))+5\\
&\leq& e(T_{n-1,3})-3\sqrt{\varepsilon}n+5< e(T_{n-1,3})+2
\end{eqnarray*}
for large enough $n,$ which contradicts \eqref{eq1}.

\begin{case}
$w_1\neq w_2$ and $u_1= v_1.$
\end{case}
In this case, we have $|S_1|\geq 2.$ Similarly to Case $1$, we have $V_2\cap L=\emptyset.$
If $u_1\notin L,$ then by Lemma \ref{lem3.11}, there exists an even path of length $2k-2$ between $w_2$ and $w_1$  such that $V(P)\subseteq N_G(u_1)\setminus \{u,v\}.$ Then we can find a $W_{2k+2}$ with center vertex $u_1$ and rim $w_1uvw_2\cdots w_1,$ a contradiction. Hence $u_1\in L.$ For any $w\in S_1,$ we claim that $N_G\big(w,(N_G(u,V_3)\cup N_G(v,V_3))\setminus \{u_1\}\big)=\emptyset$ for otherwise, it reduces to Case $1.$
Note that $N_G(u,V_3)\cap N_G(v,V_3)=\{u_1\}.$ Then
\begin{eqnarray}\label{eq31}
e_G(\{u,v,w\},V_3\setminus \{u_1\})\leq |V_3|-1.
\end{eqnarray}
Combining \eqref{eq31}, $u_1\in L$ and Lemma \ref{lem3.1}, we deduce that
\begin{eqnarray*}
e(G)&\leq& e(K_{|V_1|,|V_2|,|V_3|})-3(|V_3|-1)+e_G(\{u,v,w\},V_3\setminus \{u_1\})-(|V_1|+|V_2|)+d_G(u_1)+1\\
&\leq& e(K_{|V_1|-1,|V_2|,|V_3|})-3\sqrt{\varepsilon}n+3< e(T_{n-1,3})+2
\end{eqnarray*}
for large enough $n,$ which contradicts \eqref{eq1}.

\begin{case}
$w_1=w_2$ and $u_1=v_1.$
\end{case}

In this case, $u_1\in S_2.$ By Lemma \ref{lem3.13}, $S_2=\{u_1\}.$ We first claim that
\begin{eqnarray}\label{eq32}
e_G\big(S_1,N_G(u,V_3)\cup N_G(v,V_3)\big)=1.
\end{eqnarray}
Suppose to the contrary that there exists an edge $w'u'\in E_G\big(S_1,N_G(u,V_3)\cup N_G(v,V_3)\big)$ such that $w'u'\neq w_1u_1.$ If $w'=w_1,$ then $u'\neq u_1.$ Since $w_1\in S_1,$ $S_2=\{u_1\},$ $u'\in N_G(u,V_3)\cup N_G(v,V_3)$ and $\{w_1u_1, w_1u'\}\subset E(G),$ it reduces to Case $2.$ If $u'=u_1,$ then $w'\neq w_1.$ Since $\{w', w_1\}\subset S_1,$ $S_2=\{u_1\},$ and $\{w_1u_1, w'u_1\}\subset E(G),$ it reduces to Case $3.$ If $w'\neq w_1$ and $u'\neq u_1,$ then by $\{w', w_1\}\subset S_1,$ $S_2=\{u_1\},$ $u'\in N_G(u,V_3)\cup N_G(v,V_3)$ and $\{w_1u_1, w'u'\}\subset E(G),$ it reduces to Case $1.$ Hence \eqref{eq32} holds. Note that
\begin{eqnarray*}
e_G(\{u,v\},V_2\cup V_3)
&=& |N_G(u,V_2)\cup N_G(v,V_2)|+|S_1|+|N_G(u,V_3)\cup N_G(v,V_3)|+|S_2|\\
&\leq& |V_2|+|S_1|
   +|N_G(u,V_3)\cup N_G(v,V_3)|+1.
\end{eqnarray*}
Combining \eqref{eq32} and Lemma \ref{lem3.12}, we have
\begin{eqnarray*}
e(G) &\leq& e(K_{|V_1|,|V_2|,|V_3|})-2(|V_2|+|V_3|)+e_G(\{u,v\},V_2\cup V_3)-|S_1||N_G(u,V_3)\cup N_G(v,V_3)|\\
& & +e_G\big(S_1,N_G(u,V_3)\cup N_G(v,V_3)\big)+1\\
&\leq& e(K_{|V_1|,|V_2|,|V_3|})-2(|V_2|+|V_3|)+(|V_2|+|S_1|+|N_G(u,V_3)\cup N_G(v,V_3)|+1)\\
& &-|S_1||N_G(u,V_3)\cup N_G(v,V_3)|+2\\
&=& e(K_{|V_1|-1,|V_2|,|V_3|}) -|V_3|+4-(|S_1|-1) \big(|N_G(u,V_3)\cup N_G(v,V_3)|-1\big)\\
&\leq& e(T_{n-1,3})-|V_3|+4<e(T_{n-1,3})+2
\end{eqnarray*}
for sufficiently large $n,$ which contradicts \eqref{eq1}.
\end{proof}

\begin{lemma}\label{lem3.15}
$1\leq |S_1|\leq |V_2|-1$ and $|S_2|=1.$
\end{lemma}

\begin{proof}
By Lemma \ref{lem3.13}, $|S_1|\geq 1.$
We first prove that $|S_1|\leq |V_2|-1.$
Suppose to the contrary that $S_1=V_2.$ It follows from Lemma \ref{lem3.14} that $E(V_2,N_G(u,V_3))=\emptyset$ or $E(V_2,N_G(v,V_3))=\emptyset.$ Without loss of generality, we may assume that $E(V_2,N_G(u,V_3))=\emptyset.$ Then $G$ is a $3$-partite graph with partite sets $V_1\setminus \{u\}, V_2\cup N_G(u,V_3)$ and $(V_3\setminus N_G(u,V_3))\cup \{u\}$, a contradiction.

It follows from Lemma \ref{lem3.13} that $|S_2|\leq 1.$ Next we show that $|S_2|=1.$ Suppose to the contrary that $|S_2|=0.$ Recall that $d_G(u,V_3)\geq 1$ and $d_G(v,V_3)\geq 1.$ By Lemma \ref{lem3.14}, $E(S_1,N_G(u,V_3))=\emptyset$ or $E(S_1,N_G(v,V_3))=\emptyset.$ Without loss of generality, we assume that $E(S_1,N_G(u,V_3))=\emptyset.$ Note that $d_G(u,V_2\cup V_3)+d_G(v,V_2\cup V_3)\leq |V_2|+|V_3|+|S_1|.$ Then
\begin{eqnarray*}
e(G)&\leq& e(K_{|V_1|,|V_2|,|V_3|})-2(|V_2|+|V_3|)+e_G(\{u,v\},V_2\cup V_3)-|S_1|d_G(u,V_3)+1\\
&\leq& e(K_{|V_1|-1,|V_2|,|V_3|})-|S_1|(d_G(u,V_3)-1)+1\\
&\leq& e(T_{n-1,3})+1,
\end{eqnarray*}
which contradicts \eqref{eq1}.
\end{proof}

Now we are ready to provide the proof of Theorem \ref{main1}.

\medskip
\noindent \textbf{Proof of Theorem \ref{main1}.}
By Lemma \ref{lem3.15}, $|S_2|= 1.$ Recall that $uv\in E(G[V_1]).$ Then we obtain that
\begin{eqnarray}\label{eq33}
e_G(\{u,v\},V_2\cup V_3)\leq |V_2|+|V_3|+|S_1|+1
\end{eqnarray}
with equality if and only if $N_G(u,V_2\cup V_3)\cup N_G(v,V_2\cup V_3)=V_2\cup V_3.$
By Lemma \ref{lem3.14}, without loss of generality, we may assume that $E(S_1,N_G(u,V_3))=\emptyset.$
It follows from \eqref{eq33} that
\begin{eqnarray}\label{eq34}
e(G)&\leq& e(K_{|V_1|,|V_2|,|V_3|})-2(|V_2|+|V_3|)+e_G(\{u,v\},V_2\cup V_3)-|S_1|d_G(u,V_3)+1\nonumber\\
&\leq& e(K_{|V_1|-1,|V_2|,|V_3|})-|S_1|(d_G(u,V_3)-1)+2
\end{eqnarray}
with equality if and only if
\begin{eqnarray}\label{eq35}
G[V(G)\setminus\{u,v\}]= K_{|V_1\setminus\{u,v\}|,|V_2|,|V_3|}-E(K_{|V_1\setminus\{u,v\}|,|V_2|,|V_3|}[S_1\cup N_G(u,V_3)])
\end{eqnarray}
and equality holds in \eqref{eq33}. By $uv\in E(G[V_1])$ and the maximality of $e_G(V_1,V_2,V_3),$ we have $d_G(u,V_3)\geq 1.$ It follows from \eqref{eq34} that
\begin{eqnarray}\label{eq36}
e(G)&\leq& e(K_{|V_1|-1,|V_2|,|V_3|})-|S_1|(d_G(u,V_3)-1)+2\nonumber\\
&\leq& e(K_{|V_1|-1,|V_2|,|V_3|})+2 \leq e(T_{n-1,3})+2
\end{eqnarray}
with equality if and only if $d_G(u,V_3)=1$ and $e(K_{|V_1|-1,|V_2|,|V_3|})= e(T_{n-1,3}).$
By \eqref{eq1} and \eqref{eq36}, $e(G)=e(T_{n-1,3})+2.$ Hence $d_G(u,V_3)=1,$ $e(K_{|V_1|-1,|V_2|,|V_3|})= e(T_{n-1,3}),$
\eqref{eq35} holds, equality holds in \eqref{eq33}, and $N_G(u,V_2\cup V_3)\cup N_G(v,V_2\cup V_3)=V_2\cup V_3.$
Combining $|S_2|=1,$ $d_G(u,V_3)=1$ and $N_G(u,V_2\cup V_3)\cup N_G(v,V_2\cup V_3)=V_2\cup V_3,$ we have
\begin{eqnarray}\label{eq37}
N_G(v,V_3)=V_3.
\end{eqnarray}
Since $e(K_{|V_1|-1,|V_2|,|V_3|})= e(T_{n-1,3}),$ we have $K_{|V_1|-1,|V_2|,|V_3|}\cong T_{n-1,3}.$ So we can deduce that $|V_1|=\big\lceil\frac{n-1}{3}\big\rceil+1$ or $\big\lfloor\frac{n-1}{3}\big\rfloor+1,$ and $\big||V_2|-|V_3|\big|\leq 1.$

\begin{claim}\label{cla12}
If $N_G(v,V_2)\neq S_1,$ then $d_G(v,V_2)\leq k-1.$
\end{claim}
\begin{proof}
Suppose to the contrary that $d_G(v,V_2)\geq k.$ By Lemma \ref{lem3.15}, $|S_1|\geq 1.$
Let $v_1\in N_G(v,V_2)\setminus N_G(u,V_2),$ $v_2\in S_1,$ $\{v_3, \ldots, v_k\}\subset N_G(v,V_2)\setminus \{v_1,v_2\},$ $v'_1\in N_G(u,V_3)\cap N_G(v,V_3)$ and $\{v'_2, \ldots, v'_k\}\subset N_G(v,V_3)\setminus \{v'_1\}.$ Then $\{v'_2, \ldots, v'_k\}\cap N_G(u,V_3)=\emptyset.$ By \eqref{eq35}, we can find a $W_{2k+2}$ in $G$ with center vertex $v$ and rim $v_1v'_1uv_2v'_2\cdots v_kv'_kv_1,$ a contradiction.
\end{proof}

Next we prove that $N_G(v,V_2)=S_1.$
Suppose to the contrary that $N_G(v,V_2)\neq S_1.$ Then $N_G(v,V_2\setminus S_1)\neq \emptyset.$
By Claim \ref{cla12}, $|S_1|<d_G(v,V_2)\leq k-1.$ Let $N_G(u,V_3)=\{w\}.$ It follows from \eqref{eq35} that $N_G(w)=V_1\cup (V_2\setminus S_1).$ Then $N_G(w,V_2\setminus S_1)\cap N_G(v,V_2\setminus S_1)\neq \emptyset.$
Recall that $N_G(u,V_2\cup V_3)\cup N_G(v,V_2\cup V_3)=V_2\cup V_3.$ Since $d_G(v,V_2)\leq k-1$ and $S_1\subseteq N_G(v,V_2),$ we have $d_G(u,V_2\setminus S_1)>0.$ Hence $N_G(w,V_2\setminus S_1)\cap N_G(u,V_2\setminus S_1)\neq \emptyset.$
Let $w_1\in N_G(w,V_2\setminus S_1)\cap N_G(u,V_2\setminus S_1)$ and $w_2\in N_G(w,V_2\setminus S_1)\cap N_G(v,V_2\setminus S_1).$
Choose $k-2$ vertices $w_3,\ldots, w_k\in V_2\setminus (S_1\cup\{w_1,w_2\})$ and $k-1$ vertices $w'_3,\ldots, w'_{k+1}\in V_1\setminus \{u,v\}.$
Combining $N_G(w)=(V_1\cup V_2)\setminus S_1$ and \eqref{eq35}, we have $w_2w'_3w_3\cdots w'_kw_kw'_{k+1}w_1$ is a path in $N_G(w).$ Hence we obtain a $W_{2k+2}$ in $G$ with center vertex $w$ and rim $w_1uvw_2w'_3w_3\cdots w'_kw_kw'_{k+1}w_1,$ a contradiction.

Since $N_G(v,V_2)=S_1$ and $N_G(u,V_2\cup V_3)\cup N_G(v,V_2\cup V_3)=V_2\cup V_3,$ we have $N_G(u,V_2)=V_2$. Moreover, by \eqref{eq37}, $N_G(v,V_3)=V_3$. Since $d_G(u,V_3)=1$, let $N_G(u,V_3)=\{w\}$. Because $E_G(S_1,N_G(u,V_3))=\emptyset$, we have $E_G(S_1,\{w\})=\emptyset$.
Hence $G$ is obtained from $K_{|V_1|,|V_2|,|V_3|}$ by adding the edge $uv$ and deleting $E(V_2\setminus S_1,\{v\}),E(V_3\setminus\{w\},\{u\}),E(S_1,\{w\}).$
Since $|V_1|=\big\lceil\frac{n-1}{3}\big\rceil+1$ or $\big\lfloor\frac{n-1}{3}\big\rfloor+1,$ and $\big||V_2|-|V_3|\big|\leq 1,$ we have $G\in \{G^*(n,|S_1|): 1\leq |S_1|\leq |V_2|-1\} = \mathscr{G}(n).$ This completes the proof of Theorem \ref{main1}.
\hspace*{\fill}$\Box$

\section{Concluding remarks}
In this paper, we determined the exact Tur\'{a}n number and all extremal graphs for non-$3$-partite $W_{2k+2}$-free graphs, where $k\geq2$ and $n$ is sufficiently large. This gives a complete non-$3$-partite refinement of Tur\'{a}n theorem for the even wheel $W_{2k+2}$. One notable feature of our result is that the extremal family $\mathscr{G}(n)$ does not depend on the order of the forbidden even wheel. The reason is structural: for every graph $G\in \mathscr{G}(n)$ and every vertex $x\in V(G)$, the induced subgraph $G[N_G(x)]$ is bipartite. Hence no vertex-neighborhood of $G$ contains an odd cycle.
This observation shows that our construction of extremal graphs is stronger than what is needed for forbidding a single even wheel. Indeed, a graph contains no even wheel if and only if every vertex-neighborhood is bipartite. Since every graph in $\mathscr{G}(n)$ has this property, the graphs in $\mathscr{G}(n)$ are free of all even wheels simultaneously.

Let $\mathscr{W}_{\mathrm{even}}=\{W_{2k+2}:k\geq1\}$ be the family of all wheels with an even number of vertices. Denote by $\mathrm{ex}_4(n,\mathscr{W}_{\mathrm{even}})$ the maximum number of edges in a non-$3$-partite graph of order $n$ containing no member of $\mathscr{W}_{\mathrm{even}}$. The family of all such graphs is denoted by  $\mathrm{Ex}_4(n,\mathscr{W}_{\mathrm{even}})$. As an immediate consequence of Theorem~\ref{main1}, we obtain the following result.

\begin{theorem}\label{main2}
Let $n$ be sufficiently large. Then
\begin{eqnarray*}
\mathrm{ex}_4(n,\mathscr{W}_{\mathrm{even}})=e(T_{n,3})-\left(n-\left\lceil\frac n3\right\rceil\right)+2 \quad\text{and}\quad \mathrm{Ex}_4(n,\mathscr{W}_{\mathrm{even}})=\mathscr{G}(n).
\end{eqnarray*}
\end{theorem}

\begin{proof}
Take any $G\in \mathscr{G}(n)$. As observed above, for every vertex $x\in V(G)$, the graph $G[N_G(x)]$ is bipartite. Then $G[N_G(x)]$ contains no odd cycle, and so $G$ contains no even wheel. Note that $G$ is non-$3$-partite. For any $G^*\in \mathrm{Ex}_4(n,\mathscr{W}_{\mathrm{even}}),$ we have
\begin{eqnarray}\label{eq38}
e(G^*)\geq e(G)=e(T_{n,3})-\left(n-\left\lceil\frac n3\right\rceil\right)+2.
\end{eqnarray}
Since $G^*\in \mathrm{Ex}_4(n,\mathscr{W}_{\mathrm{even}}),$ $G^*$ is a non-$3$-partite $W_6$-free graph. Applying Theorem \ref{main1} with $k=2$, we deduce that
\begin{eqnarray}\label{eq39}
e(G^*)\leq e(T_{n,3})-\left(n-\left\lceil\frac n3\right\rceil\right)+2
\end{eqnarray}
with equality if and only if $G^*\in\mathscr{G}(n)$. Combining \eqref{eq38} and \eqref{eq39}, we have
\begin{eqnarray*}
e(G^*)= e(T_{n,3})-\left(n-\left\lceil\frac n3\right\rceil\right)+2,
\end{eqnarray*}
and hence $\mathrm{ex}_4(n,\mathscr{W}_{\mathrm{even}})= e(T_{n,3})-\left(n-\big\lceil\frac n3\big\rceil\right)+2.$ Moreover, equality holds in \eqref{eq39}, and hence $G^*\in \mathscr{G}(n),$ which implies that $\mathrm{Ex}_4(n,\mathscr{W}_{\mathrm{even}})\subseteq \mathscr{G}(n).$
Conversely, every graph in $\mathscr{G}(n)$ is a non-$3$-partite $\mathscr{W}_{\mathrm{even}}$-free graph and has $\mathrm{ex}_4(n,\mathscr{W}_{\mathrm{even}})$ edges. Then $\mathscr{G}(n) \subseteq \mathrm{Ex}_4(n,\mathscr{W}_{\mathrm{even}}).$
Hence $\mathrm{Ex}_4(n,\mathscr{W}_{\mathrm{even}})=\mathscr{G}(n).$
\end{proof}

In view of Theorems \ref{main1} and \ref{main2}, we not only solve the Tur\'{a}n problem for a fixed forbidden even wheel $W_{2k+2}$ in non-$3$-partite graphs, but also settle the stronger Tur\'{a}n-type problem in which all even wheels are forbidden simultaneously.

Finally, set $N(k,\varepsilon)=\max\left\{n_1(W_{2k+2}), \Big\lceil\frac{4}{3\delta}\Big\rceil, \Big\lceil\frac{3}{8\varepsilon}\Big\rceil\right\},$ where $\delta$ and $n_1(W_{2k+2})$ are the constants provided by Lemma~\ref{lem2.1} for the chosen $\varepsilon$. One can see from the proof of Theorems \ref{main1} and \ref{main2} that the results in these two theorems hold as long as $n\geq N(k,\varepsilon).$

\vspace{5mm}
\noindent
{\bf Declaration of competing interest}
\vspace{3mm}

The authors declare that they have no known competing financial interests or personal relationships that could have appeared to influence the work reported in this paper.

\vspace{5mm}
\noindent
{\bf Data availability}
\vspace{3mm}

No data was used for the research described in this paper.

\vspace{5mm}
\noindent
{\bf Acknowledgement}

\vspace{3mm}
The first author gratefully acknowledges the support of the China Scholarship Council and the hospitality of the School of Mathematics and Statistics at the University of Melbourne.

\end{document}